\documentclass[11pt,a4paper]{article}
\usepackage[utf8]{inputenc}

\usepackage[all]{xy}
\usepackage{latexsym}
\usepackage[margin=3cm]{geometry}
\usepackage[dvipdfmx]{graphicx}

\usepackage[
backend=biber,
style=ext-alphabetic,
maxalphanames = 5,
maxcitenames  = 5,
minalphanames = 5,
minnames      = 3,
sorting=nyt,
giveninits,
articlein=false,
url=false,  
doi=false,
eprint=true,
isbn=false
]{biblatex}
\usepackage{hyperref}
\usepackage{mymacro}
\usepackage{mathtools}

\renewcommand{\labelenumi}{$\mathrm{(\roman{enumi})}$}

\SelectTips{cm}{}

\makeatletter

\newcommand{\ostar}{\mathbin{\mathpalette\make@circled\star}}
\newcommand{\make@circled}[2]{%
  \ooalign{$\m@th#1\smallbigcirc{#1}$\cr\hidewidth$\m@th#1#2$\hidewidth\cr}%
}
\newcommand{\smallbigcirc}[1]{%
  \vcenter{\hbox{\scalebox{0.77778}{$\m@th#1\bigcirc$}}}%
}

\makeatother

\title{Heavy subsets from microsupports}
\author{Tomohiro Asano}
\date{\today}

\begin{document}
\maketitle

\begin{abstract}
    We construct partial symplectic quasi-states on a cotangent bundle with the use of microlocal sheaf theory. 
    We also give criteria and characterization for heaviness/superheaviness with respect to the partial symplectic quasi-state.
\end{abstract}

\tableofcontents

\section{Introduction}

Let $X$ be a symplectic manifold. 
For closed subsets $A,B\subset X$, $A$ is said to be non-displaceable from $B$ if there exists no Hamiltonian diffeomorphism $\varphi\colon X\to X$ with $\varphi(A)\cap B=\emptyset$. 
The topic of non-displaceability has studied over the years from various viewpoints. 

Entov--Polterovich~\cite{EP06quasi-state,EP09rigid} introduced the notion of partial symplectic quasi-states and heavy/superheavy subsets. 
A partial symplectic quasi-state is a special functional defined on the function space of a symplectic manifold. 
For a partial symplectic quasi-state the notions of heavy/superheavy subsets are assigned. 
These notions provide a powerful tool to prove non-displaceability.

Tamarkin~\cite{Tamarkin} introduced triangulated categories $h\cT (T^*M), h\cT_\infty (T^*M)$ and proved non-displaceability results. 
The category $h\cT (T^*M)$ is a subcategory of the derived category of the sheaves on $M\times \bR$ and $h\cT_\infty(T^*M)$ is a quotient category of $h\cT (T^*M)$. 
For an object $F$ of $h\cT (T^*M)$, a closed subset $\RS (F)\subset T^*M$ called the reduced microsupport of $F$ is defined. 
Tamarkin proved that the non-tiriviality of the morphism space between $F$ and $G$ in $h\cT_\infty(T^*M)$ implies non-displaceability of $\RS(F)$ and $\RS(G)$ if they are compact.
See \cref{subsec:Tamarkin_cat,subsec:spectral}. 

In this paper, we construct partial symplectic quasi-states via sheaves. 
We also give criteria for heaviness/superheaviness with respect to these partial symplectic quasi-states. 

\subsection{Main results}

We first define partial symplectic quasi-states from sheaves.

\begin{theorem}[{see \cref{theorem:idempotent_quasistate,prop:existence_superheavy}}]
    One can construct a partial symplectic quasi-state $\zeta_F$ from a non-trivial object $F\in h\cT_\infty (T^*M)$. 
    Moreover, if $\RS(F)$ is compact, it is $\zeta_F$-superheavy.
\end{theorem}

We can extend this construction to more general idempotents on $F$. 
In this introduction, we restrict ourselves to the cases that the idempotent is the identity morphism $\id_F$ for simplicity. 
For $F=\bfk_{M\times [0,\infty)}$, $\zeta_F$ coincides with the partial symplectic quasi-state $\zeta_{\mathrm{MVZ}}$ defined via the Lagrangian intersection Floer theory by \cite{MVZ12}. 

We then give a sufficient condition for $\zeta_F$-heaviness. 
For objects $F, G$ of $h\cT (T^*M)$, the subset $\Spec(F,G)$ of $\bR$ is defined and reflects some non-tiriviality of the morphism space between $F$ and $G$ in $h\cT_\infty(T^*M)$. 
\begin{theorem}[{see \cref{theorem:zetae_heavy}}]
    Let $G \in h\cT(T^*M)$.
    Assume that $\RS(G)$ is compact and the set $\Spec(F, G)$ of the spectral invariants is non-empty.
    Then $\RS(G)$ is $\zeta_{F}$-heavy.
\end{theorem}

We also give sufficient conditions that the reduced microsupport is $\zeta_{\mathrm{MVZ}}$-superheavy (see \cref{corollary:superheavyV} and \cref{theorem:superheavy}).

For $F$'s satisfying some conditions,
we also give a characterization for $\zeta_F$-heaviness.

\begin{theorem}[{see \cref{theorem:characterization}}]
    For a special $F$, the following conditions for a compact subset $A\subset T^*M$ are equivalent:  
    \begin{enumerate}
    \item $A$ is $\zeta_F$-heavy,
    \item there exist an object $G\in h\cT(T^*M)$ with $\RS(G)\subset A$ and a non-zero morphism $F\to G$ in $h\cT_\infty(T^*M)$, 
    \item the unit morphism $\eta_{A,F} \colon F\to  {\iota_A}_*\iota_A^*F$ of the adjunction $\iota_A^*\dashv {\iota_A}_*$ satisfies $\eta_{A,F} \neq 0$ as a morphism of $h\cT_\infty(T^*M)$. 
\end{enumerate}
\end{theorem}
For the proof, we use the explicit description of the functor $\iota_A^*$ by Kuo~\cite{Kuowrap}.

With these results and known results about Viterbo's spectral bound conjecture, we give several examples for $\zeta_{\mathrm{MVZ}}$-superheavy subsets, which include generalization of some of the results in Kawasaki--Orita~\cite{KO22rigidfiber}.
Furthermore, in conjunction with recent studies on $\gamma$-support~\cite{AGHIV,AHV24}, we see that $\zeta_{\mathrm{MVZ}}$-superheavy subsets can become indecomposable continua.

We also construct partial symplectic quasi-states with a direct use of sheaf quantization of Hamiltonian diffeomorphisms and prove characterization of heaviness in \Cref{sec:psqsHam}. 

\subsection{Related work}

There are two kinds of well-studied partial symplectic quasi-states on the cotangent bundles of compact manifolds. 
Lanzat~\cite{Lanzat13} defined partial symplectic quasi-states for Liouville domains using spectral invariants of Hamiltonian Floer homologies.
Monzner--Vichery--Zapolsky~\cite{MVZ12} defined a partial symplectic quasi-state on a cotangent bundle using spectral invariants of Lagrangian Floer homologies for the zero-section.

Chracterizations of heaviness already appeared in \cite{OnoSugimoto, MSV23}. 
Ono--Sugimoto~\cite{OnoSugimoto} constructed partial symplectic quasi-states for symplectic manifolds with boundaries of contact type from idempotents of symplectic cohomology and gave a characterization of heaviness.
Mak--Sun--Varolgunes~\cite{MSV23} also proved a similar charaterization of heviness with respect to  partial symplectic quasi-states defined via quantum cohomologies of closed symplectic manifolds. 
Both studies use relative symplectic cohomologies to determine the heaviness. 
The main result of \cite{OnoSugimoto} restricted for cotangent bundles corresponds to \cref{theorem:Hamchar} of this paper.  
See \cref{remark:OSrem} for a difference.

Since the relation between sheaf kernels and symplectic cohomologies of domains in cotangent bundles are described in \cite{KSZ23} depending on a result of \cite{guillermou2022gamma}, the author expects that more concrete relation between our work and Ono--Sugimoto's would be revealed.

\subsection*{Acknowledgements}
The author would like to thank Yuichi Ike for his numerous helpful comments and suggestions regarding the structure of this paper.
The author is also thankful to Yoshihiro Sugimoto for his proposal, which inspired the content of \cref{section:characterization}.
The author would also like to thank Morimichi Kawasaki for enlightening him about the connection with Viterbo's conjecture.  
Furthermore, the author also thank Christopher Kuo, Tatsuki Kuwagaki, Wenyuan Li, and Ryuma Orita for helpful comments. 

This work was supported by JSPS KAKENHI Grant Number JP24K16920.

\section{Preliminaries}\label{section:preliminaries}

In this paper, we work in the area of microlocal sheaf theory developed by Kashiwara--Schapira~\cite{KS90} with the framework of infinity categories. 
For the basics of infinity categories, refer to \cite{LurieHTT,LurieHA}. 
Regarding the implementation of microlocal sheaf theory within the setting of infinity categories, 
we will follow Kuo's setting~\cite{Kuowrap} in this paper.
See \cite{Kuowrap} and its references.

Throughout this paper, we fix a coefficient field $\bfk$.
For a manifold $X$, 
let $\Sh (X)$ be the $\bfk$-linear stable derived category of sheaves of $\bfk$-vector spaces on $X$. 
For each object $F\in \Sh(X)$, we write $\MS(F)\subset T^*X$ for the microsupport of $F$, which is a closed conic subset.
For a closed subset $A\subset T^*X$, $\Sh_A(X)$ denotes the full subcategory of $\Sh(X)$ consisting of objects whose microsupports are contained in $A$.

\subsection{Tamarkin category}\label{subsec:Tamarkin_cat}

Until the end of this paper, let $M$ be a connected manifold without boundary. 
We write $(t;\tau)$ for the homogeneous coordinate system on the cotangent bundle $T^*\bR$.
The Tamarkin category $\cT(T^*M)$ is defined to be $\Sh (M\times \bR)/\Sh_{\{\tau \le 0 \}}(M\times \bR)$. 
The homotopy category $h\cT(T^*M)$ is isomorphic to the quotient category $h\Sh (M\times \bR)/h\Sh_{\{\tau \le 0 \}}(M\times \bR)$ of the homotopy categories by \cite[Prop.~5.9]{BGT13}. 
For an object $F\in \cT (T^*M)$, $\MS(F)\cap \{\tau>0 \}$ is invariant under isomorphisms in $\cT(T^*M)$. 
See \cite{Tamarkin,GS14} for details. 
For an object $F \in \cT(T^*M)$, we define 
\begin{equation}
    \RS(F) \coloneqq \overline{\rho(\MS(F) \cap \{ \tau>0 \})},
\end{equation}
where $\rho \colon T^*_{\tau>0}(M \times \bR_t) \to T^*M; (x,t;\xi,\tau) \mapsto (x;\xi/\tau)$.
The closed subset $\RS(F)$ is called the \emph{reduced microsupport} of $F \in \cT(T^*M)$.
For a closed subset $A\subset T^*M$, let $\cT_A(T^*M)$ denote the full subcategory of $\cT (T^*M)$ consisting of objects $F$ with $\RS (F)\subset A$.

We introduce notation for basic operations in the Tamarkin categories.
We set 
\begin{equation}
\begin{aligned}
    & q_{1} \colon M \times \bR \times \bR \to M \times \bR; 
    (x,t_1,t_2) \mapsto (x,t_1) \\
    & q_{2} \colon M \times \bR \times \bR \to M \times \bR; 
    (x,t_1,t_2) \mapsto (x,t_2) \\
    & m \colon M \times \bR \times \bR \to M \times \bR; 
    (x,t_1,t_2) \mapsto (x,t_1+t_2)
\end{aligned}
\end{equation}
and define 
\begin{equation}
    F \star G \coloneqq {m}_!(q_1^{-1} F \otimes q_2^{-1} G) \in \Sh(\bfk_{M \times \bR})
\end{equation}
for $F,G \in \Sh(\bfk_{M \times \bR})$.
This induces a functor $\star\colon \cT(T^*M)\times \cT(T^*M)\to \cT(T^*M)$, which is called the convolution functor and provides $\cT(T^*M)$ with a monoidal structure. 
The monoidal structure is closed, and the monoidal unit is given by $\bfk_{M\times [0,\infty)}$. 
The internal homomorphism is given by $\cHom^\star$, which is defined as 
\begin{equation}
    \cHom^\star(G,H) \coloneqq {q_1}_* \cHom(q_2^{-1}G,m^!H).
\end{equation}

Set $q_{\bR}\colon M\times \bR \to \bR$ be the projection. 

\begin{lemma}\label{lemma:TamarkinVerdier}
    Let $F, G \in \Sh(\bfk_{M \times \bR})$.
    Assume $F$ is cohomologically constructible and the projection $q_{\bR}|_{\Supp(F)}\colon \Supp(F) \to \bR$ is proper and has a bounded image.
    Then, one has an isomorphism
    \begin{equation}
        \cHom^\star(F,G) \simeq D(i^{-1}F) \star G,
    \end{equation}
    where $D$ denotes the Verdier dual and $i \colon M \times \bR \to M \times \bR; (x,t) \mapsto (x,-t)$.
    In particular, the functor $\cHom^\star(F,\mathchar`-)$ preserves colimits.    
\end{lemma}
\begin{proof}
    First note that we have an isomorphism 
    \begin{equation}
        \cHom^\star(F,G) \simeq m_* \cHom(q_2^{-1}i^{-1}F,q_1^!G).
    \end{equation}
    Since $i^{-1}F$ is cohomologically constructible, by \cite[Prop.~3.4.4]{KS90}, 
    \begin{equation}
        \cHom(q_2^{-1}i^{-1}F,q_1^!G) \simeq G \boxtimes D(i^{-1}F).
    \end{equation}
    By the assumption, $m$ is proper on the support of $G \boxtimes D(i^{-1}F)$, which proves the first isomorphism.
    The second assertion follows from the fact that the functor $D(i^{-1}F) \star (\mathchar`-)$ is a left adjoint.
\end{proof}

\begin{remark}
    In the proof, the assumption about the boundedness of the image of $q_{\bR}|_{\Supp(F)}\colon \Supp(F) \to \bR$ is only used for the properness of $m$ on $\Supp(G \boxtimes D(i^{-1}F))$. 
    This properness is satisfied also in other situations:
    \begin{enumerate}
        \item $q_{\bR}|_{\Supp(G)}\colon \Supp(G) \to \bR$ has a bounded image, 
        \item The image of $q_{\bR}|_{\Supp(F)}$ is bounded below and that of $q_{\bR}|_{\Supp(G)}$ is bounded above,
        \item The image of $q_{\bR}|_{\Supp(F)}$ is bounded above and that of $q_{\bR}|_{\Supp(G)}$ is bounded below.
    \end{enumerate}
    Hence one can replace the boundedness of the image of $q_{\bR}|_{\Supp(F)}$ by either of them. 
    We do not use the other versions in this paper. 
\end{remark}

We will use a variant of the convolution functor.
Let $M_i \ (i=1,2,3)$ be a manifold. 
We set 
\begin{equation}
\begin{aligned}
    q_{12} \colon M_1 \times M_2 \times M_3 \times \bR \times \bR & \to M_1 \times M_2 \times \bR; \\ 
    (x_1,x_2,x_3,t_1,t_2) & \mapsto (x_1,x_2,t_1) \\
    q_{23} \colon M_1 \times M_2 \times M_3 \times \bR \times \bR & \to M_2 \times M_3 \times \bR; \\
    (x_1,x_2,x_3,t_1,t_2) & \mapsto (x_2,x_3,t_2) \\
    m_{13} \colon M_1 \times M_2 \times M_3 \times \bR \times \bR & \to M_1 \times M_3 \times \bR; \\
    (x_1,x_2,x_3,t_1,t_2) & \mapsto (x_1,x_3,t_1+t_2).
\end{aligned}
\end{equation}
With these maps, for $F_{12} \in \Sh(\bfk_{M_1 \times M_2 \times \bR})$ and $F_{23} \in \Sh(\bfk_{M_2 \times M_3 \times \bR})$, we define 
\begin{equation}
    F_{12} \ostar F_{23} \coloneqq {m_{13}}_! (q_{12}^{-1} F_{12} \otimes q_{12}^{-1} F_{23}) \in \Sh(\bfk_{M_1 \times M_3 \times \bR}),
\end{equation}
which induces a functor $\ostar \colon \cT (T^*M_1\times T^*M_2)\times \cT (T^*M_2\times T^*M_3)\to \cT (T^*M_1\times T^*M_3)$.

For $c\in \bR$, $T_c\colon M\times \bR\to M\times \bR\colon (x,t)\mapsto (x,t+c)$.
We shall abbreviate the functor ${T_c}_*$ as $T_c$. 
For $c\ge 0$, there exists a natural transformation $\tau_c\colon \id_{\cT(T^*M)} \Rightarrow T_c$.
Asano--Ike~\cite{AI20} defined a pseudo-distance on the class of the objects of $\cT(T^*M)$ with the use of $T_c$ and $\tau_c$, following the sheaf-theoretic interleaving distance by Kashiwara--Schapira~\cite{KS18persistent}.
In this paper, we use the following pseudo-distance (cf.~\cite{AI22completeness}).

\begin{definition}
    Let $F,G \in \cT(T^*M)$ and $a,b \ge 0$.
    \begin{enumerate}
        \item The pair $(F,G)$ is said to be \emph{$(a,b)$-isomorphic} if there exist morphisms $\alpha \colon F \to T_a G$ and $\beta \colon G \to T_b F$ in $h\cT(T^*M)$ such that 
        \[
        \begin{cases}
            \ld F \xrightarrow{\alpha} T_a G \xrightarrow{T_a \beta} T_{a+b} F \rd = \tau_{a+b}(F), \\
            \ld G \xrightarrow{\beta} T_b F \xrightarrow{T_b \alpha} T_{a+b} G \rd = \tau_{a+b}(G).
        \end{cases}
        \]
        \item We define 
        \[
            d_{\cT(T^*M)}(F,G) \coloneqq \inf \lc a+b \relmid \text{$(F,G)$ is $(a,b)$-isomorphic} \rc.
        \]
        \item One defines $\mathrm{Tor}$ to be the full triangulated subcategory of $h\cT(T^*M)$ consisting of the objects $F$ with $d_{\cT(T^*M)}(F,0)<+\infty$. 
        The quotient category $h\cT(T^*M)/\mathrm{Tor}$ is denoted by $h\cT_\infty(T^*M)$.
    \end{enumerate}
\end{definition}

The Hom space in $h\cT_\infty(T^*M)$ is described as follows. 

\begin{lemma}[{\cite[Prop.~5.7]{GS14}}]\label{proposition:HomhT}
    For $F, G \in \cT(T^*M)$, one has
    \begin{equation}
        \Hom_{h\cT_\infty (T^*M)}(F,G)\simeq \colim_c \Hom_{h\cT(T^*M)} (F,T_c G).
    \end{equation}
\end{lemma}

\begin{remark}
    The notation $h\cT_\infty (T^*M)$ suggests that this category is the homotopy category of a certain stable category $\cT_\infty (T^*M)$. 
    Indeed, the author expects that such $\cT_\infty (T^*M)$ is obtained as a quotient of $\cT (T^*M)$ by the full subcategory of torsion objects. 
    This fact seems to follow from, for example, the appendix of \cite{Kuowrap} or \cite[Thm.~I.3.3]{NS18}, where the morphisms of the quotient of a small stable categories are described via colimits, and the fact that these colimits can be replaced by the smaller colimits as discussed in \cite{GS14}.
    Precisely speaking, we need to work in a larger universe so that $\cT(T^*M)$ become a small category.
    Regardless, we only need the triangulated category $h\cT_\infty (T^*M)$, so we do not go into this aspect in this paper.
\end{remark}

\subsection{Sheaf quantization of Hamiltonian diffeomorphisms and homeomorphisms}\label{subsec:SQHam}

In this subsection, we recall sheaf quantization of Hamiltonian isotopies due to Guillermou--Kashiwara--Schapira~\cite{GKS}.

First we introduce several classes of functions following Lanzat~\cite{Lanzat13} and Ono--Sugimoto~\cite{OnoSugimoto}. 
We let $C^\infty_c(T^*M)$ denote the set of compactly supported $C^\infty$-functions on $T^*M$. 
Furthermore, we define 
\begin{equation}
    C^\infty_{cc}(T^*M)\coloneqq \{ H\colon T^*M\to \bR \mid \text{there exists } C_H\in \bR \text{ such that } H-C_H\in C^\infty_c(T^*M) \}. 
\end{equation}
We also write $C_c^\infty(T^*M\times [0,1])$ for the set of compactly supported $C^\infty$-functions on $T^*M \times [0,1]$, and set
\begin{equation}
    C_{cc}^\infty(T^*M \times [0,1]) \coloneqq \{ H \in C^\infty(M \times [0,1]) \mid \text{$H|_{T^*M\times \{s\}}\in C_{cc}^\infty (T^*M)$ for each $s\in [0,1]$} \}.
\end{equation}
If $M$ is $0$-dimensional, we define $C^\infty_c(T^*M)\coloneqq \{ 0\}$ and $C^\infty_c(T^*M\times [0,1])\coloneqq \{ 0\}$ exceptionally.
We say that $H\in C_{cc}^\infty(T^*M)$ or $H\in C_{cc}^\infty(T^*M\times [0,1])$ is normalized if it is in $H\in C_{c}^\infty(T^*M)$ or $C_{c}^\infty(T^*M\times [0,1])$.

Each $H\in C_{cc}^\infty(T^*M\times [0,1])$ generates an isotopy $\phi^H=(\phi^H_s\colon T^*M\to T^*M)_{s\in [0,1]}$ called a Hamiltonian isotopy. 
Let $\Ham_c(T^*M)$ be the set $\{\phi^H_1 \mid H\in C_{cc}^\infty(T^*M\times [0,1])\}$ of the time $1$-maps of the Hamiltonian isotopies, which forms a subgroup of the diffeomorphism group of $T^*M$. 
For $H,H'\in C_{cc}^\infty(T^*M\times [0,1])$, put $H\sharp H'(p,s)\coloneqq H(p,s)+H'((\phi^H_s)^{-1}(p),s)$ and $\overline{H}(p,s)\coloneqq-H(\phi^H_s(p),s)$. 
Note that $\phi^{H\sharp H'}_s=\phi^H_s\circ \phi^{H'}_s$ and $\phi^{\overline{H}}_s=(\phi^H_s)^{-1}$. 

For $H\in C_{cc}^\infty(T^*M \times [0,1])$, one can uniquely associate an object $\tl{\cK}_H\in \cT (T^*M^2\times T^*[0,1])$, which is called the sheaf quantization or the Guillermou--Kashiwara--Schapira (GKS) kernel \cite{GKS}.
See also \cite[Part~2]{Gu23cotangent} and \cite{KSZ23}. 
We set $\cK_H\coloneqq \tl{\cK}_H|_{M^2\times \{1\}\times \bR} \in \cT(T^*M\times T^*M)$.

For $\varphi \in \Ham_c(T^*M)$, there exists a normalized Hamiltonian function $H$ with $\phi^H_1=\varphi$. 
We define $\cK_\varphi\coloneqq \cK_H$. 
The following lemma asserts that this assignment is well-defined. 

\begin{lemma}[{\cite[Prop.~5.9]{AI22completeness}}]
    Let $H$ and $H'$ be normalized time-dependent Hamiltonian functions with $\phi^H_1=\phi^{H'}_1$. 
    Then $\cK_H\simeq \cK_{H'}$. 
\end{lemma}

The following lemmas are basic properties of the GKS kernel. 
\begin{lemma}\label{lemma:GKSsharp}
    For $H,H'\in C_{cc}^\infty (T^*M \times [0,1])$, $\cK_{H\sharp H'}\simeq \cK_H \ostar \cK_{H'}$. 
\end{lemma}

\begin{lemma}\label{lemma:GKSconst}
    Let $h\colon [0,1]\to \bR$ be a smooth function and $s\colon T^*M\times [0,1]\to [0,1]$ be the projection. 
    For the timewise constant function $H= h\circ s \in C_{cc}^\infty (T^*M \times [0,1])$, $\cK_H\simeq \bfk_{\Delta_M \times [c,\infty)}$ where $c=\int_0^1 h(s)\; ds$. 
\end{lemma}
\begin{lemma}
    For $H\le H' \in C_{cc}^\infty (T^*M \times [0,1])$, there exists a natural morphism $\cK_H\to \cK_{H'}$ which corresponds to $1\in H^0(M)\simeq \Hom_{h\cT_\infty (T^*M^2)}(\cK_H, \cK_{H'})$. 
\end{lemma}

For the pseudo-distance $d_{\cT(T^*M)}$, \cite{AI20,AI22completeness} proved the following Hamiltonian stability theorem.

\begin{theorem}[{\cite[Thm.~4.16]{AI20} and \cite[Thm.~5.1]{AI22completeness}}]\label{theorem:Ham_stability}
    For $H\in C_c^\infty (T^*M\times [0,1])$, one has 
    \begin{equation}
        d_{\cT(T^*M^2)}(\cK_0, \cK_H)\le \| H\|_{\mathrm{osc}},
    \end{equation}
    where 
    \begin{equation}
        \|H\|_{\mathrm{osc}} \coloneqq \int_0^1 \lb \max_{p \in T^*M} H_s(p) - \min_{p \in T^*M} H_s(p) \rb ds.
    \end{equation}
\end{theorem}

We define $C_c(T^*M), C_{cc}(T^*M), C_c(T^*M\times [0,1])$ and $C_{cc}(T^*M \times [0,1])$ by replacing $C^\infty$-functions by continuous functions. 
Moreover we regard $C_c(T^*M)\subset  C_c(T^*M\times [0,1])$ and $C_{cc}(T^*M)\subset C_{cc}(T^*M \times [0,1])$. 
We have the direct sum decompositions
\begin{align}
    C_{cc}(T^*M ) & = C_{c}(T^*M ) \oplus \bR, \\
    C_{cc}(T^*M \times [0,1]) & =   C_{c}(T^*M \times [0,1])\oplus C([0,1]). 
\end{align}
Let $H\in C_{cc}(T^*M\times [0,1])$. 
We set $H_{\mathrm{norm}} \in C_c(T^*M \times [0,1])$ to be the normalized part through the above decomposition and write $\Supp(H)\subset T^*M$ for the projection of the support of $H_{\mathrm{norm}}$ under the projection $T^*M\times [0,1]\to T^*M$. 
Note that $\Supp (H)\subset T^*M$ is compact. 

It is proved in \cite{AI22completeness,guillermou2022gamma} that the pseudo-distance $d_{\cT(T^*X)}$ is complete, i.e., any Cauchy sequence converges. 
By combining the completeness and the Hamiltonian stability (\cref{theorem:Ham_stability}), we can associate a GKS kernel for continuous Hamiltonian functions as follows.
Let $H \in C_{cc}(T^*M \times [0,1])$.
Then there exists a sequence $(H_n)_n \subset C^\infty_{cc}(T^*M \times [0,1])$ satisfying $\| H_n - H \|_{C^0} \to 0$ as $n \to \infty$.  
Since $(H_n)_n$ is a Cauchy sequence with respect to $\| \cdot \|_{\mathrm{osc}}$, the sequence $(\cK_{H_n})_n$ is also a Cauchy sequence with respect to $d_{\cT(T^*M^2)}$ by \cref{theorem:Ham_stability}.
By the completeness of $d_{\cT(T^*M^2)}$, this sequence converges to an object of $\cT(T^*M^2)$. 
Moreover, by \cite[Prop.~B.7]{guillermou2022gamma}, the object is unique. 
We write $\cK_H \in \cT(T^*M\times T^*M)$ for the object. 
Note that $(\cK_{\overline{H_n}})_n$ also converges and 
we write $\cK_H^{-1} \in \cT(T^*M\times T^*M)$ for the object.

In the later applications, we let $\cK_H^{\ostar k}$ denote $\underbrace{\cK_H \ostar \dots \ostar \cK_H}_{\text{$k$-times}}$ for $k \in \bN_{\ge 1}$. 
We extend the notation as $\cK_H^{\ostar 0} \coloneqq \bfk_{\Delta \times [0,\infty)}$ and $\cK_H^{\ostar -k} \coloneqq (\cK^{-1}_H)^{\ostar k}$. 
Furthermore, for $K\in \cT (T^*M\times T^*M)$, the functor $K \ostar (\mathchar`-)$ is also denoted as $K^{\ostar}(\mathchar`-)$, and $(\cK^{\ostar k}_H)^{\ostar}(\mathchar`-)$ is abbreviated as $\cK^{\ostar k}_H(\mathchar`-)$.

Let $\cC_{cc} (T^*M\times [0,1])$ (resp.\ $\cC_{cc}^\infty (T^*M\times [0,1])$) be the nerve of the poset $C_{cc}(T^*M\times [0,1])$ (resp.\ $C_{cc}^\infty (T^*M\times [0,1])$).

\begin{lemma}[cf.~{\cite[Sec.~3.3]{Kuowrap}}]\label{lemma:infintyGKS}
    The assignment of the GKS kernels is refined to an infinity functor $\cC_{cc} (T^*M\times [0,1])\to \cT(T^*M^2)$. 
\end{lemma}
\begin{proof}
    The functor $\cC_{cc}^\infty (T^*M\times [0,1])\to \cT(T^*M^2)$ is defined by a similar argument to \cite[Sec.~3]{Kuowrap}. 
    By the left Kan extension along the inclusion $\cC_{cc}^\infty (T^*M\times [0,1])\to \cC_{cc}(T^*M\times [0,1])$, we obtain a functor $\cC_{cc} (T^*M\times [0,1])\to \cT(T^*M^2)$.
    
    This construction is compatible with the construction above since homotopy colimits \cite{bokstedt1993homotopy} used in the previous papers \cite{AI22completeness, guillermou2022gamma} corresponds to colimits in the infinity categorical sense, which appear in the explicit description of the left Kan extension. 
    Precisely speaking, we need to be careful to see the correspondence of the colimits. 
    Note that for $H\in C_{cc} (T^*M\times [0,1])$, the poset $\{H'\in C_{cc}^\infty (T^*M\times [0,1]) \mid H'\le H\}$ may not be filtered. 
    For $H, H'\in C_{cc} (T^*M\times [0,1])$, we write $H'\prec H$ if there exists a real number $a<0$ satisfying $H'\le H+a$. 
    For $H\in C_{cc} (T^*M\times [0,1])$, the poset $\{H'\in C_{cc}^\infty (T^*M\times [0,1]) \mid H'\prec H\}$ is filtered and the colimit indexed by this poset corresponds to the homotopy colimit in \cite{AI22completeness, guillermou2022gamma}. 
    The colimits indexed by $2$ posets $\{H'\in C_{cc}^\infty (T^*M\times [0,1]) \mid H'\prec H\}$ and $\{H'\in C_{cc}^\infty (T^*M\times [0,1]) \mid H'\le H\}$ are isomorphic since $\colim_{a<0}\cK_{H'+a}\simeq \cK_{H'}$ holds for any $H'\in C_{cc}^\infty (T^*M\times [0,1])$ by \cref{lemma:GKSsharp,lemma:GKSconst}. 
\end{proof}

\subsection{An explicit description of projectors}\label{subsec:explicitadjoints}

Let $A\subset T^*M$ be a compact subset.
The inclusion functor $\iota_{A*} \colon \cT_A(T^*M)\to \cT(T^*M)$ admits a left adjoint $\iota_A^*$ and a right adjoint $\iota_A^!$ by the presentavility of $\cT_A(T^*M)$ and $\cT(T^*M)$.
In this section, we give explicit descriptions for $\iota_A^*$ and $\iota_A^!$. 

Kuo's explicit descriptions~\cite{Kuowrap} of the adjoint functors of the inclusion functor $\Sh_{\rho^{-1}(A)\cup 0_{M\times \bR}}(M\times \bR)\to \Sh (M\times \bR)$ descend to $\cT(T^*M) \to \cT_A(T^*M)$, where $0_{M\times \bR}$ is the image of the zero section of $T^*(M\times \bR)$. 
In this paper, we use a variant of his description, which is also used in \cite{KSZ23}. 

For a compact subset $A\subset T^*M$, 
define 
\begin{align}
    C_{cc}(T^*M\times [0,1],A) & \coloneqq \{H \in C_{cc}(T^*M\times [0,1])\mid \text{$H \equiv 0$ on a neighborhood of $A$} \}, \\
    C_{cc}(T^*M,A) & \coloneqq \{H \in C_{cc}(T^*M)\mid \text{$H \equiv 0$ on a neighborhood of $A$} \}.
\end{align}
Define $C_{cc}^\infty (T^*M\times [0,1],A)$ and $C_{cc}^\infty(T^*M,A)$ similarly. 

Let $\cC_{cc}(T^*M,A)$ (resp.\ $\cC_{cc}(T^*M\times [0,1],A)$, $\cC_{cc}^\infty(T^*M,A)$, $\cC_{cc}^\infty(T^*M\times [0,1],A)$) be the nerve of the poset $C_{cc}(T^*M,A)$ (resp.\ $C_{cc}(T^*M\times [0,1],A), C_{cc}^\infty(T^*M,A), C_{cc}^\infty(T^*M\times [0,1],A)$). 

Kuo~\cite{Kuowrap} considers compactly supported Hamiltonians on the sphere cotangent bundle, whereas we and Kuo--Shende--Zhang~\cite{KSZ23} consider Hamiltonians on $J^1(M)=T^*M\times\bR$ that are translation invariant, and become compactly supported after taking the quotient of the domain by the $\bR$-action. 
This is the only difference, and the proofs proceed parallelly.
Hence we obtain the following descripitions of the adjoint functors of ${\iota_{A}}_*$
\begin{proposition}[{cf.\ \cite[Thm.~1.2]{Kuowrap} and \cite[Cor.~6.6]{KSZ23}}]
    \begin{equation}
        \iota^!_A F =\lim_{H\in \cC_{cc}(T^*M,A)} \cK_H\ostar F, \quad
        \iota^*_A F =\colim_{H\in \cC_{cc}(T^*M,A)} \cK_H\ostar F. 
    \end{equation}
\end{proposition}

\begin{remark}
    Every inclusion functor between any two of $\cC_{cc}(T^*M,A)$, $\cC_{cc}(T^*M\times [0,1],A)$, $\cC_{cc}^\infty(T^*M,A)$, $\cC_{cc}^\infty(T^*M\times [0,1],A)$ is initial and final if it exists. 
    So we can replace the index category of the limit or the colimit by other ones. 
\end{remark}

Define
\begin{equation}
    \cK_A\coloneqq \colim_{H\in \cC_{cc}(T^*M,A)} \cK_H \in \cT (T^*M^2). 
\end{equation}
Since $\ostar$ commutes with colimits, $\iota_A^* F $ is equivalent to $\cK_A \ostar F$. 
Since $({\iota_A}_*\iota_A^*)^2\simeq {\iota_A}_*\iota_A^*$,  
\begin{equation}\label{equation:kernelidempotnce}
    \cK_A\ostar \cK_A \simeq \cK_A
\end{equation}
by \cite[Prop.~5.12]{KSZ23}.

\section{Partial symplectic quasi-states from sheaves}\label{section:quasistate}

In this section, we construct a partial symplectic quasi-state from an idempotent of a sheaf category.
From now on, we assume the compactness of $M$. 
Hence $M$ is a connected closed manifold. 

\subsection{Spectral invariants in Tamarkin category}\label{subsec:spectral}

In this subsection, we introduce spectral invariants for objects in the Tamarkin category.
See also Vichery~\cite{VicheryPhD}, Asano--Ike~\cite{AI22completeness}, and  Guillermou--Viterbo~\cite{guillermou2022gamma} for example.

\begin{definition}
    Let $F, G \in \cT(T^*M)$.
    One defines 
    \begin{align}
        h\cT_c^d(F,G) & \coloneqq \Hom_{h\cT(T^*M)} (F,T_cG[d]), \\
        h\cT_c^*(F, G) & \coloneqq \bigoplus_{d\in \bZ} h\cT_c^d(F, G), \\
        h\cT_\infty^d(F,G) & \coloneqq \colim_c h\cT_c^d(F,G)\simeq\Hom_{h\cT_\infty(T^*M)}(F, G[d]), \\
        h\cT_\infty^*(F,G) & \coloneqq \colim_c h\cT_c^*(F,G) \simeq \bigoplus_{d\in \bZ} \Hom_{h\cT_\infty(T^*M)}(F, G[d]).
    \end{align}
\end{definition}

\begin{definition}
    Let $F, G \in \cT(T^*M)$. 
    For $\alpha \in h\cT^*_\infty(F,G)$, one sets 
    \begin{equation}
        c(\alpha;F,G) 
        \coloneqq 
        -\inf\{ c \in \bR \mid \alpha \in \Image(h\cT^*_c(F,G) \to h\cT^*_\infty(F,G)) \}\in \bR\cup \{ +\infty \}
    \end{equation}
    and 
    \begin{equation}
        \Spec(F,G) \coloneqq \{ c(\alpha;F,G)\in \bR \mid \alpha \in h\cT^*_\infty(F,G) \} \subset \bR.
    \end{equation}
\end{definition}

Note that our $\Spec(F, G)$ never contains $+\infty$ by definition.

\begin{remark}
    The sign convention of spectral invariants differs among papers. In this paper, we follow that of Monzner--Vichery--Zapolsky~\cite{MVZ12}. This is consistent with the convention of Guillermou--Viterbo~\cite{guillermou2022gamma}, but it is opposite to that of Oh~\cite{Oh97,Oh99}, Frauenfelder--Schlenk~\cite{FS07} and the previous work~\cite{AI22completeness}.
    See also \cite[Section~2.4.2]{MVZ12}. 
\end{remark}

For $\alpha \in h\cT^d_\infty(F,G ) $, if we can take a lift $\tilde{\alpha} \colon F \to T_c G[d]$ of $\alpha$, then we have
\begin{equation}
    -c \le c(\alpha;F, G).
\end{equation}

 \begin{remark}\label{remark:c=+infty}
     In our definition, $c(0;F,G)=+\infty$. Generally, $c(\alpha; F,G)$ can be $+\infty$ for non-zero $\alpha \in h\cT_{\infty}^*(F,G)$. 
     See \cite[Rem.~6.2]{AI22completeness} for an example. 
\end{remark}

For $F,G\in \cT(T^*M)$, we define 
\begin{equation}
    \Sigma(F,G):=-q_\bR\pi((\MS(G)\widehat{\star}\MS(F)^\alpha ) \cap \Gamma_{dt}). 
\end{equation}
See \cite{GS14} for the definitions of $\widehat{\star}$ and $(\mathchar`-)^\alpha$. 
However, we do not need to refer to these definitions directly; for the necessary properties of $\Sigma(F,G)$, see the following.
By \cite{GS14}, 
\begin{equation}
    \MS(\cHom^\star(F,G))\subset \RS(G)\widehat{\star}\RS(F)^\alpha 
\end{equation}
and its immediate consequence is the following. 
\begin{lemma}\label{lemma:spectrality}
    For any $\alpha \in h\cT^*_\infty(F,G)$, 
    \begin{equation}
        c(\alpha; F, G)\in \Sigma (F,G)\cup \{+\infty \}. 
    \end{equation}
\end{lemma}
\begin{remark}\label{remark:spectrality}
    The calculation of $\Sigma(F, G)$ is not easy in general. 
    However we will use this result only for either case of the following:
    \begin{enumerate}
   \renewcommand{\labelenumi}{$\mathrm{(\arabic{enumi})}$}
        \item $\RS (F)$ or $\RS(G)$ is compact,
        \item both of $F$ and $G$ are of the form $\cK_H$ for some $H\in C_{cc}^{\infty}(T^*M\times [0,1])$. 
    \end{enumerate}
    For both cases, it is easy to see
    \begin{equation}
        \Sigma (F,G)=\{c\in \bR \mid \MS(T_cF)\cap \MS(G)\cap \{\tau>0\}\neq \emptyset\}. 
    \end{equation}
\end{remark}

In favorable situations, phenomena like those mentioned in \cref{remark:c=+infty} do not occur.

\begin{lemma}\label{lemma:specinfty}
    Let $F, G \in \cT(T^*M)$ and assume $\Sigma(F,G)$ is bounded above or below.
    If $\alpha \in h\cT_\infty(F,G)$ satisfies $c(\alpha;F,G)=\infty$, then $\alpha=0$.
\end{lemma}
\begin{proof}
    Let us consider a morphism $\tl{\alpha} \colon F\to G$ in $\cT (T^*M)$ which is a lift of $\alpha$ with $c(\alpha ;F,G)=+\infty $. 
    For any $s\in \bR_{\ge 0}$, there exists a lift $\tl{\alpha}_s\colon T_{s}F\to G$ of $\alpha$. 
    There exists $a_s\le 0$ such that $\tl{\alpha}_s$ and $\tl{\alpha}$ induces the same morphism $T_{a_s}F\to G$. 
    This implies the natural morphism $F\star\bfk_{M\times [a_s,s)}\to G$ induced by $\tl{\alpha}$ is $0$.    
    
    Assume $\Sigma(F,G)\subset (c,\infty)$ for some $c\in \bR$, then we can take any $c'\in (-\infty, c)$ as $a_s$ above. 
    Hence $F\star\bfk_{M\times [c',s)} \to G$ induced by $\tl{\alpha}$ is zero. 
    Since $\colim_s F\star\bfk_{M\times [c',s)}\simeq T_{c'}F$, the morphism $\tl{\alpha}\colon T_{c'}F\to G$ is also zero.  
    This shows $\alpha=0$. 
    
    Asuume $\Sigma(F,G) \subset (-\infty ,c)$ for some $c\in \bR$, then the canonical morphism $h\cT_{-c'}^*(F,G)\to h\cT_{-c}^*(F,G)$ is isomophic for every $c'\ge c$. 
    Hence any morphism $\tl{\alpha}_{c}\colon T_cF\to G[d]$ uniquely factors through a morphism $\tl{\alpha}_{c'}\colon T_{c'} F\to G[d]$. 
    By the uniqueness, we obtain a morphism $\tl{\alpha}_\infty\colon \colim_{c'} T_{c'}F\to G[d]$. 
    For every object $F\in \cT(T^*M)$, $\colim_{c'} T_{c'}F\simeq (\colim_{c'} \bfk_{[c',\infty)} )\star F\simeq 0$ holds. 
    Then we obtain $\tl{\alpha}_\infty =0$ and hence $\tl{\alpha}_{c}=0$ by $\Hom (\colim_{c'} T_{c'}F, G[d])\simeq \lim_{c'} \Hom (T_{c'}F, G[d])\simeq \Hom (T_cF, G[d])$. 
\end{proof}

We will use the following basic property of the spectral invariant. 

\begin{lemma}\label{lemma:specprodineq}
    One has
    \begin{equation}
        c(\beta \alpha; F,F'' ) \ge c(\alpha; F,F')+c(\beta; F',F''). 
    \end{equation}
\end{lemma}

\subsection{Sheaf-theoretic spectral norm}

For $\phi, \psi \in \Ham_c(T^*M)$, we set
\begin{equation}
    \gamma (\phi, \psi)\coloneqq -c(1;\cK_\phi, \cK_\psi )-c(1;\cK_\psi , \cK_\phi ). 
\end{equation}
We write $\gamma (\phi)\coloneqq \gamma (\phi, \id )$ for short. 
We call $\gamma (\phi)$ the sheaf-theoretic spectral norm of $\phi$.

\begin{remark}
    With the use of a version of Hamiltonian Floer theory, Frauenfelder--Schlenk~\cite{FS07} defined a spectral norm $\gamma^{\mathrm{Floer}}(\phi)$ for each compactly supported Hamiltonian diffeomorphism $\phi$. 
    Guillermou--Viterbo~\cite[Thm.~E.1]{guillermou2022gamma} combined with Viterbo's earlier result~\cite[Lem.~3.2]{viterbo2018functors} proves the filtered isomorphism between Hamiltonian Floer cohomology and some sheaf cohomology defined from GKS kernels if the characteristic of the coefficient field is $2$ or if the base manifold $M$ is spin. 
    Hence $\gamma^{\mathrm{Floer}}$ and $\gamma$ in this paper coincide for these cases.
\end{remark}

We give some basic properties on $\gamma$, which is similar to those for $\gamma^{\mathrm{Floer}}$ in \cite{FS07}. 

\begin{lemma}[cf.~{\cite[Prop.~7.4]{FS07}}]\label{lemma:displace_spectral}
    Let $H,H'\in C_{c}^\infty (T^*M\times [0,1])$ and assume $\phi_{H'}$ displaces $\Supp (H)$ then
    \begin{equation}
        \gamma(\phi_H) \le 2\gamma (\phi_{H'}).
    \end{equation}
\end{lemma}

To give a sheaf-theoretic proof of this lemma, we prepare the spectrality of the spectral invariants for GKS kernels and 
geometric description of $\Sigma (\cK_H,\cK_{H'})$

By \cref{lemma:spectrality,lemma:specinfty,remark:spectrality}, we obtain the following lemma. 
\begin{lemma}\label{lemma:Hamspectrality} 
    One has
    \begin{equation}
        c(\alpha; \cK_H, \cK_{H'})\in \Sigma (\cK_H,\cK_{H'})=\{c\in \bR \mid \MS(T_c \cK_H)\cap \MS(\cK_{H'})\cap \{\tau>0\}\neq \emptyset\}
    \end{equation}
    for any non-zero element $\alpha \in h\cT^*_\infty (\cK_H,\cK_{H'})\simeq H^*(M)$.
\end{lemma}

For $H\in C_{cc}^\infty (T^*M\times [0,1])$, define 
\begin{equation}
    \Sigma(H)\coloneqq \left\{\int_0^1 (H_s-\alpha (X_{H_s}))( \gamma(s) )\;ds  \relmid \gamma\colon [0,1]\to T^*M, \dot{\gamma}(s)=X_{H_s}(\gamma (s)), \gamma(0)=\gamma(1) \right\}.
\end{equation}

By a direct calculation of $\MS (\cK_H)$, we obtain the following lemma.  
\begin{lemma}\label{lemma:SigmaGKS}
    One has
    \begin{equation}\label{eq:SigmaGKS}
        \Sigma(\cK_0, \cK_H)=\Sigma(H).
    \end{equation}
\end{lemma}
\begin{corollary}\label{corollary:Sigmanondense}
    The set $\Sigma(\cK_\phi, \cK_\psi)$ is nowhere dense in $\bR$ for every $\phi, \psi \in \Ham_c(T^*M)$.
\end{corollary}
\begin{proof}
    It is enough to prove $\Sigma(H)$ is nowhere dense for each $H\in  C_{c}^\infty (T^*M)$. 
    We may assume $H$ descends to a smooth function on $T^*M\times \bR/\bZ$ by a reparameterization. 
    By embedding $M$ into a Euclidean space $\bR^m$, we can construct a compactly supported smooth function $\tl{H}\colon T^*\bR^m\times \bR/\bZ\to\bR$ with $\Sigma(H) \subset \Sigma(\tl{H})$. 
    By \cite[Chap.~5, Prop.~8]{HZ11}, $\Sigma(\tl{H})$ is nowhere dense and hence, so is $\Sigma(H)$. 
\end{proof}

\begin{proof}[Proof of \cref{lemma:displace_spectral}]
    We will mimic the proof by \cite[Prop.~7.4]{FS07}. 
    Define $c_+(H'')\coloneqq -c(1; \cK_{H''}, \cK_0)$ and $\gamma (H)\coloneqq \gamma (\phi_{H''})$ for each $H''\in C_{cc}^{\infty} (T^*M\times [0,1])$. 
    
    By reparameterizing, we may assume that $H_s=0$ for $s \in [0,1/2]$ and $H'_s=0$ for $s \in [1/2,1]$. 
    We can check $\Sigma(\cK_{0}, \cK_{\varepsilon H \sharp H'}) = \Sigma(\cK_{0}, \cK_{H'})$ for $\varepsilon \in [0,1]$.
    Since $\Sigma(\cK_{0}, \cK_{H'}) = \Sigma(\cK_{0}, \cK_{\varepsilon H \sharp H'})$ is nowhere dense, by the continuity of $c_+$ we find that the map 
    \begin{equation}
        [0,1] \to \Sigma(\cK_{0}, \cK_{H'}); \quad 
        \varepsilon \mapsto c_+(\varepsilon H \sharp H')
    \end{equation}
    is constant. 
    In particular, $c_+(H \sharp H')=c_+(H')$. 
    Since $\phi^{H'}_1$ displaces $\supp H$, its inverse $\phi^{\overline{H'}}_1$ displaces $\supp \overline{H}=\supp H$. 
    An argument similar to the above shows that
    \begin{equation}
        c_+(\overline{H \sharp H'})
        =
        c_+(\overline{H'} \sharp \overline{H})
        = 
        c_+(\overline{H'}).
    \end{equation}
    Hence, we have $\gamma(H \sharp H') = \gamma(H')$. 
    We thus conclude that 
    \begin{equation}
        \gamma(H) = \gamma(H \sharp H' \sharp \overline{H'})
        \le \gamma(H \sharp H') + \gamma(\overline{H'}) = 2\gamma(H').
    \end{equation}
\end{proof}

\begin{lemma}\label{lemma:cnonposi}
    The inequality $c(1;\phi,\psi) \le 0$ holds for each $\phi, \psi$. 
\end{lemma}

The following proof is obtained through communication with Yuichi Ike. 

\begin{proof}
    Since $c(1;\phi,\psi)=c(1;\id,\phi^{-1}\psi)$, it is enough to show that $c(1;\id,\phi^H_1) \le 0$ for any $H \in C_{c}^\infty(T^*M\times [0,1])$.
    First let us study the case $H$ of the form $H=f(|\xi|)$, where $f$ is a $C^\infty$-function on $[0,\infty)$ such that, for some $0< C< C'$ and $b>0$,      
    \begin{enumerate}
    \renewcommand{\labelenumi}{$\mathrm{(\arabic{enumi})}$}
        \item $f \equiv b>0$ on $\{ z < C \}$,
        \item $f \equiv 0$ on $\{ z > C' \}$,
        \item $0 \le f \le b$, 
        \item $f$ is strictly decreasing and $\frac{df}{dz}$ is non-zero on $C < z < C'$,
        \item $\frac{df}{dz}$ is sufficiently small so that $\varphi^H$ has no non-trivial fixed points.
    \end{enumerate}
    By microsupport estimate, we can see that $\Hom_{h\cT(T^*M)}(\cK^0, T_c \cK^H)$ only changes at $c=0$ or $-b$.
    In each case, the change can be estimated by the microlocal stalk 
    \begin{equation}
        \mu_0 (q_* \cHom^\star(\cK_0, T_c \cK_H)),
    \end{equation}
    where $q \colon M^2 \times \bR \to \bR$ is the projection.
    Although the intersection of  $\RS(\cK_0)$ and $\RS(\cK_H)$ is not clean, we can apply an argument similar to \cite{Ike19} to see that it is isomorphic to
    \begin{equation}
        \RG(\{ \tau >0 \} ; \muhom(\cK_0,T_c \cK_H)|_{\{ \tau >0 \}}).
    \end{equation}
    See for example \cite{KS90,Gu23cotangent,Ike19} for the definition and treatment of $\muhom$. 
    Calculations similar to \cite[Appendix A]{Ike19} show
    \begin{equation}
         \muhom(\cK_0, \cK_H)|_{\{ \tau >0 \}} \simeq \bfk_{\{ |\xi/\tau| \ge C' \}}
    \end{equation}
    and 
    \begin{equation}
         \muhom(\cK_0,T_{b} \cK_H)|_{\{ \tau >0 \}} \simeq \bfk_{\{ |\xi/\tau|<C \}}.
    \end{equation}
    The contribution at $c=0$ is $H^*(S^*M)$ and that at $c=-b$ is $H^*(M;\mathrm{or}_M)[-n]$. 
    We obtain a long exact sequence 
    \begin{equation}\label{eq:Gysin?}
        H^*(M;\ori_M)[-n] \to H^*(M) \to H^*(S^*M) \xrightarrow{+1}.
    \end{equation} 
    Since $H^d(M)[-n]$ is zero for $d<n$, $H^0(M)\simeq \bfk$ in the middle term injects into $H^*(S^*M)$. 
    This means that $c(1;\id,\phi^H_1) = 0$.

    For a general $H \in C_{c}^\infty(T^*M\times [0,1])$, we can take $H'=f(|\xi|)$ with $f$ satisfying the above conditions such that $H \le H'$.
    By the monotonicity of $c$, we have $c(1;\id,\phi^H_1) \le c(1;\id,\phi^{H'}_1)$.
    The right-hand side is 0 by the above argument.
\end{proof}
\begin{remark}
    In this proof, we did not identified the morphisms in the exact sequence (\ref{eq:Gysin?}) since it is not needed for the proof. 
    The author conjectures that it coincides with the Gysin exact sequence. 
\end{remark}

\begin{proposition}[{\cite[Prop.~6.14]{guillermou2022gamma}}]\label{proposition:gamma=d}
    One has
    \begin{equation}
        \gamma(\phi,\psi)=d_{\cT(T^*M^2)}(\cK_\phi, \cK_{\psi}).
    \end{equation}
\end{proposition}
\begin{proof}
     With \cref{lemma:cnonposi} in hand, the rest of the discussion is parallel to those of \cite{AI22completeness} and \cite{guillermou2022gamma}. 
\end{proof}

\subsection{Partial symplectic quasi-states from sheaves}\label{subsec:partial_quasistate}

In this subsection, we construct a partial symplectic quasi-state from an idempotent of $h\cT_\infty(T^*M)$.

First, let us recall the definition of partial symplectic quasi-states on $T^*M$.
The notion of partial symplectic quasi-states was first introduced by Entov--Polterovich~\cite{EP09rigid} on closed symplectic manifolds. 
Later, Lanzat~\cite{Lanzat13} and Monzner--Vichery--Zapolsky~\cite{MVZ12} defined  partial symplectic quasi-states for some open symplectic manifolds, especially for the cotangent bundles of closed manifolds.

\begin{definition}\label{definition:partial_quasistate}
    A \textit{partial symplectic quasi-state} $\zeta$ is a map $\zeta \colon C_{cc}(T^*M)\to \bR$ satisfying the following conditions:
    \begin{enumerate}
    \renewcommand{\labelenumi}{$\mathrm{(\arabic{enumi})}$}
        \item (Lipschitz continuity) For all $H_1, H_2 \in C_{cc}(T^*M)$, $|\zeta (H_1)-\zeta (H_2)|\le \|H_1-H_2\|_{C^0}$.
        \item (Semi-homogeneity) For all $\lambda \ge 0$ and $H \in C_{cc}(T^*M)$, $\zeta (\lambda H)=\lambda \zeta (H)$.
        \item (Monotonicity) If $H_1, H_2 \in C_{cc}(T^*M)$ and $H_1\le H_2$, then $\zeta (H_1)\le \zeta (H_2)$.
        \item (Additivity with respect to constants and Normalization) For all $H \in C_{cc}(T^*M)$ and $a \in \bR$, $\zeta (H+a)=\zeta (H)+a$.
        \item (Partial additivity) If $H_1, H_2 \in C_{cc}^\infty (T^*M)$ and $\{H_1,H_2\}=0$, then $\zeta (H_1+H_2)\le \zeta (H_1)+\zeta (H_2)$. 
        \item (Vanishing) If $H \in C_{c} (T^*M)$ and $\Supp (H)$ is displaceable from itself, then $\zeta (H)= 0$. 
        \item (Invariance) For each $\varphi\in \Ham_c(T^*M)$ and $H \in C_{cc}(T^*M)$, $\zeta (H\circ \varphi)=\zeta (H)$.
    \end{enumerate}
\end{definition}

\begin{definition}
    Let $\zeta \colon C_{cc}(T^*M) \to \bR$ be a partial symplectic quasi-state. 
    A compact subset $A$ of $T^*M$ is said to be $\zeta$-heavy (resp.\ $\zeta$-superheavy) if for any $H \in C_{cc}(T^*M)$ one has $\zeta(H) \ge \min_{A} H$ (resp.\ $\zeta(H) \le \max_{A} H$). 
\end{definition}

We will use the following equivalent definition. 

\begin{lemma}\label{lemma:heavyequiv}
     A compact subset $A$ of $T^*M$ is $\zeta$-heavy (resp.\ $\zeta$-superheavy) if and only if for any $H \in C_{cc}(T^*M)$ and $c\in \bR$ satisfying $H\equiv c$ on a neighborhood of $A$ one has $\zeta(H) \ge c$ (resp.\ $\zeta(H) \le c$). 
\end{lemma}
\begin{proof}
    Since the ``only if" part is obvious, let us prove the ``if" part. 
    For arbitrary $H\in C_{cc}(T^*M)$, take $H'\in C_{cc}(T^*M)$ with $H'|_{A}\equiv c=\min_A H$ (resp.\ $\max_A H$) and $H\ge H'$ (resp.\ $H\le H'$). 
    For example, $\min\{c, H\}$ (resp.\ $\max \{c, H\}$) satisfies the condition for $H'$. 
    Then by the monotonicity, we get $\zeta(H) \ge \zeta(H')$ (resp.\ $\zeta(H) \le \zeta(H')$). 
    There exists a sequence $(H'_n)_n$ $C^0$-converging to $H'$ such that $H'_n|_{U_n}\equiv c$ for a neighborhood $U_n$ of $A$ for each $n$. 
    By the Lipschitz continuity, we obtain $\zeta(H')=\lim_n \zeta(H'_n)$. 
    By assumption, $\zeta(H'_n)\ge c$ (resp.\ $\zeta(H'_n)\le c$) and hence $\zeta (H)\ge c$ (resp.\ $\zeta (H)\le c$). 
\end{proof}

The following are basic and important properties of heavy/superheavy subsets. 

\begin{proposition}[{\cite[Thm.~1.4]{EP09rigid}}]\label{proposition:heavy_nondisplace}
    Let $\zeta$ be a partial symplectic quasi-state.
    \begin{enumerate}
        \item Every $\zeta$-superheavy subset is $\zeta$-heavy. 
        \item Every $\zeta$-heavy subset is non-displaceable from itself.
        \item Every $\zeta$-heavy subset is non-displaceable from every $\zeta$-superheavy subset.
    \end{enumerate}
\end{proposition}

We can define a partial symplectic quasi-state from sheaves as follows.
Let $F \in \cT(T^*M)$ and $e\colon F\to F$ be an idempotent in $h \cT_\infty(T^*M)$ satisfying $c(e;F,F)\neq +\infty$. 
We define $\zeta_e\colon C_{cc}(M)\to \bR$ by
\begin{equation}
    \zeta_e(H) \coloneqq -\lim_k\frac{c(e;F, \cK_{H}^{\ostar -k} (F))}{k}.
\end{equation}
We also define $\zeta_F$ as $\zeta_{\id_F}$ for $F\in \cT(T^*M)$ with $c(\id_F;F,F)\neq \infty$. 

\begin{remark}\label{remark:finitelength}
    Let $F, F' \in \cT (T^*M)$ be objects which are isomorphic in $h\cT_{\infty}(T^*M)$. 
    Then $d_{\cT(T^*M)} (F,F') <+\infty $ and an isomorphism $F\simeq F'$ in $h\cT_{\infty}(T^*M)$ induces an isomorphism $h\cT_{\infty}^*(F, F) \simeq h\cT_{\infty}^*(F', F')$. 
    An idempotent $e\colon F\to F$ in $h\cT_{\infty}(T^*M)$ corresponds to an idempotent $e'\colon F'\to F'$ through the isomorphism.
    The condition $c(e;F,F)<+\infty$ implies $c(e';F',F')<+\infty$ since $d (F,F') <+\infty$. 
    The difference $|c(e;F, \cK_{H}^{\ostar -k} (F))-c(e';F', \cK_{H}^{\ostar -k} (F'))|$ is bounded above by $2d_{\cT(T^*M)} (F,F')$, and hence $\zeta_e=\zeta_{e'}$. 
\end{remark}

\begin{theorem}\label{theorem:idempotent_quasistate}
    The map $\zeta_e$ is a partial symplectic quasi-state. 
\end{theorem}

To prove the theorem, we shall prepare a lemma.
Define 
\begin{equation}
    \ell_e (\phi)\coloneqq -c(e;\cK_{\phi}^{\ostar} (F), F). 
\end{equation}

\begin{lemma}\label{lemma:elle}
The following hold:
\begin{enumerate}
    \item $\ell_e (\phi \psi)\le \ell_e (\phi)+\ell_e(\psi)$,
    \item $|\ell_e(\phi)-\ell_e(\psi)|\le \gamma (\phi,\psi)$.
\end{enumerate}
\end{lemma}
\begin{proof}
    (i) is a direct consequence of \cref{lemma:specprodineq} and $e^2=e$. 

    (ii) $|\ell_e(\phi)-\ell_e(\psi)|\le d_{\cT(T^*M^2)}(\cK_\phi, \cK_{\psi})$ is deduced from a basic property of the spectral invariant. Hence the assertion follows from \cref{proposition:gamma=d}
\end{proof}

\begin{proof}[Proof of \cref{theorem:idempotent_quasistate}]
    We have prepared \cref{lemma:displace_spectral,lemma:elle}. 
    The rest of the proof is parallel to that of \cite{MVZ12}. 
\end{proof}

If $\RS(F)$ is compact, the partial symplectic quasi-state $\zeta_e$ satisfies the following important property.  
\begin{proposition}\label{prop:existence_superheavy}
    If $\RS (F)$ is compact, then $\RS (F)$ is $\zeta_e$-superheavy.
\end{proposition}

This is a straightforward consequence of the following standard lemma, to which similar statements and proofs can be found in several articles. 

\begin{lemma}[{cf.~\cite[Lem.~3.1]{AGHIV}, \cite[Thm.~A.2]{AI22completeness}, and \cite[Prop.~3.33]{Kuowrap}}]\label{lemma:constantshift} 
    Let $F\in \cT (T^*M)$ and $H\in C_{cc}(T^*M\times [0,1])$. Assume $H-c \in  C_{cc}(T^*M\times [0,1], \RS(F) )$ for some constant $c\in \bR$. Then $\cK_H^{\ostar}(F)\simeq T_c F$.
\end{lemma}
\begin{proof}
    It is enough to prove for $H\in C_{cc}^\infty(T^*M\times [0,1])$ since the kernel $\cK_H$ for continuous $H$ is defined as a colimit of $\cK_{H'}$ for smooth $H'$ and the colimit commute with $\ostar$.  
    \cite[Lem.~3.1]{AGHIV} is the statement for $c=0$. 
    The desired statement is obtained from it since $\cK_{H}\simeq T_c\cK_{H-c}$ by \cref{lemma:GKSsharp,lemma:GKSconst}. 
\end{proof}

As a special case, we can take $F=\bfk_{M \times [0,\infty)}$ and $e=\id_F$. 
In this case, if $\bfk=\bZ/2\bZ$, $\zeta_e$ coincides with the partial symplectic quasi-state $\zeta_{\mathrm{MVZ}}$ defined by Monzner--Vichery--Zapolsky~\cite{MVZ12}. 
Hence, we still denote the partial symplectic quasi-state by $\zeta_{\mathrm{MVZ}}$ for a general coefficient.

For the later application, we describe $\zeta_{\mathrm{MVZ}}$ in a different form.

\begin{definition}
    One defines
    \begin{equation}
        \ell_-(H)\coloneqq c(1;\bfk_{M\times [0,\infty)}, \cK_H^{\ostar} (\bfk_{M\times [0,\infty)})),\quad \ell_+(H)\coloneqq -c(1;\bfk_{M\times [0,\infty)}, \cK_H^{\ostar -1} (\bfk_{M\times [0,\infty)})).
    \end{equation}
\end{definition}

\begin{remark}\label{remark:ell+}
    We have the following:
    \begin{enumerate}
        \item $\ell_+(H)= -\ell_-(\overline{H})$ for $H\in C_{cc}^\infty(T^*M\times [0,1])$,
        \item $\ell_+(H)= -\ell_-(-H)$ for $H\in C_{cc}(T^*M)$,
        \item $\ell_+(H)=c(\mu_M; \bfk_{M\times [0,\infty)}, \cK_H^{\ostar} (\bfk_{M\times [0,\infty)})\otimes \mathrm{or}_{M\times \bR})$ for $H\in C_{cc}(T^*M\times [0,1])$. 
    \end{enumerate}
    The last equation follows from the work~\cite{Vi19} by Viterbo. 
    Although an orientation for $M$ is assumed in that work, parallel arguments can be applied in the absence of orientation once one use the orientation sheaf.
\end{remark}

\begin{lemma}
    For $H \in C_{cc}(T^*M)$, one has 
    \begin{equation}
        \zeta_{\mathrm{MVZ}}(H) = \lim_{k \to \infty} \frac{\ell_+(kH)}{k}=\lim_{k \to \infty} \frac{c(\mu_M; \bfk_{M\times [0,\infty)}, \cK_H^{\ostar k} (\bfk_{M\times [0,\infty)})\otimes \mathrm{or}_{M\times \bR})}{k}.
    \end{equation}
\end{lemma}

\section{Criteria for heviness/superheaviness}\label{section:results}

In this section, we give sufficient conditions for $\zeta_e$-heaviness and $\zeta_{\mathrm{MVZ}}$-superheaviness. 

\subsection{Criteria for $\zeta_e$-heaviness}

Let $F \in \cT(T^*M)$ and $e\colon F\to F$ be an idempotent in $h\cT_\infty(T^*M)$ with $c(e;F,F)\neq +\infty$. 
Let $\zeta_e$ be the partial symplectic quasi-state associated with $e$. 
In this section, we first prove the following theorem.

\begin{theorem}\label{theorem:zetae_heavy}
    Let $G \in \cT(T^*M)$. Assume that that $\RS(G)$ is compact.  
    \begin{enumerate}
        \item If there exists an element $\beta \in hT^* (F,G)$ satisfying $c(\beta e ;F,G)\neq +\infty $, then $\RS (G)$ is $\zeta_e$-heavy. 
        \item If there exists an element $\beta \in  hT^* (G,F)$ satisfying $c(e \beta ;F,G)\neq +\infty$, then $\RS (G)$ is $\zeta_e$-heavy. 
    \end{enumerate}
\end{theorem}
\begin{proof}
    We only give a proof for (i) since the other is parallel. 
    By \cref{lemma:heavyequiv}, it is enough to prove $\zeta_e(H)\ge c $ for $H \in C_{cc}(T^*M)$ with $H|_{U}\equiv c$ where $U$ is a neighborhood of $\RS(G)$. 
    
    Take $k\in \bZ$ and set $\ell \coloneqq \ell_e(kH)\coloneqq -c(e;F, \cK_{-kH}^{\ostar}(F))=-c(e;F, \cK_H^{\ostar -k}(F))$.
    For any $\varepsilon>0$, we can take a lift 
    \begin{equation}
        \tl{e} \colon F \to T_{\ell+\varepsilon}\cK_H^{\ostar -k}(F)
    \end{equation}
    of $e \in Q_\infty^*(F,F)$.
    We may assume there exist $d\in \bZ$ such that $\beta \in Q^d_\infty(F,G)$. Put $b\coloneqq c(\beta e; F,G)\in \bR$.
    For any $\varepsilon'>0$, we can take a lift 
    \begin{equation}
        \tl{\alpha} \colon F \to T_{-b+\varepsilon'} G[d]
    \end{equation}
    of $\beta e$.
    By the assumption of $H$ and \cref{lemma:constantshift}, we find that 
    \begin{equation}
        \cK_H^{\ostar -k}(G) \simeq T_{-kc}G \quad (k \in \bZ). 
    \end{equation}
    Hence, by composing the above morphisms, we get a morphism
    \begin{equation}
        F \xrightarrow{\tl{e}} T_{\ell+\varepsilon}\cK_H^{\ostar -k}F \xrightarrow{\cK_H^{\ostar -k} (\tl{\alpha})} T_{\ell-b+\varepsilon+\varepsilon' - kc} G[d],
    \end{equation}
    which is also a lift of $\beta e \in Q^*_\infty(F,G)$.
    This means that 
    \begin{equation}
        -\ell+b-\varepsilon-\varepsilon' + kc \le c(\beta e;F, G)=b.
    \end{equation}
    Since $\varepsilon$ and $\varepsilon'$ are arbitrary, we obtain 
    \begin{equation}
        \ell_e(kH) = \ell \ge kc.
    \end{equation}
    By the definition of $\zeta_e$ and taking the limit $k\to +\infty$,
    we get $\zeta_{e}(H)\ge c$. 
\end{proof}

\begin{remark}
    Let us assume $\RS(F)$ and $\RS(G)$ are compact. 
    Under the assumption of (i) or (ii) in \cref{theorem:zetae_heavy}, $\Hom_{h\cT_\infty(T^*M)}(F,G[d])$ or  $\Hom_{h\cT_\infty(T^*M)}(G,F[d])$ for some $d \in \bZ$ is non-zero. 
    Hence, by Tamarkin's non-displaceability theorem (see \cite[Thm.~3.1]{Tamarkin} and \cite[Thm.~7.2 and Cor.~7.3]{GS14}), we find that $\RS(G)$ is non-displaceable from $\RS(F)$. 
    Since any $\zeta_e$-heavy subset is non-displaceable from any $\zeta_e$-superheavy subset (\cref{proposition:heavy_nondisplace}), one can view \cref{theorem:zetae_heavy} as a refinement of Tamarkin's non-displaceability theorem. 
\end{remark}

\begin{corollary}\label{corollary:heavy}
    Let $G \in \cT(T^*M)$. Assume that $\RS(G)$ is compact.  
    \begin{enumerate}
        \item If $\Spec(\bfk_{M\times [0,\infty)}, G) \neq \emptyset$, 
    then $\RS(G)$ is $\zeta_{\mathrm{MVZ}}$-heavy.
        \item If $\Spec(G,\bfk_{M\times[0,\infty)}) \neq \emptyset$,
    then $\RS(G)$ is $\zeta_{\mathrm{MVZ}}$-heavy.
    \end{enumerate}
\end{corollary}

\begin{remark}
    We can state our results in a different way by introducing a $1$-category $\overline{h\cT_\infty}^*(T^*M)$ as follows.
    
    We set $h\cT_{\infty, -\infty}^d(F,G) \coloneqq \{\alpha \in h\cT_\infty^d(F,G) \mid c(\alpha;F,G)=+\infty \}$. 
    This is closed under precomposition and postcomposition of morphisms in $h\cT_\infty(T^*M)$.
    Let $\overline{h\cT_\infty}(T^*M)$ be the category that has the same object as $h\cT_\infty(T^*M)$ and the morphism spaces is defined by 
    \begin{equation}
        \Hom_{\overline{h\cT_\infty}(T^*M)}(F,G)\coloneqq h\cT_{\infty}^0(F,G) /h\cT_{\infty, -\infty}^0(F,G). 
    \end{equation}
    The triangulated structure on $h\cT_\infty(T^*M)$ is revoked by this quotient.  
    Define 
    \begin{align}
        \overline{h\cT_\infty}^d(F,G)& \coloneqq h\cT_{\infty}^d(F,G) /h\cT_{\infty, -\infty}^d(F,G) =\Hom_{\overline{h\cT_\infty}(M)}(F,G[d]), \\
        \overline{h\cT_\infty}^*(F,G)&\coloneqq \bigoplus_{d\in \bZ}\overline{h\cT_\infty}^d(F,G).
    \end{align}
    Notice that the assignment $c( \mathchar`-; F,G)\colon h\cT_\infty^*(F,G) \to \bR\cup \{ +\infty \}$ descends to an assignment $\overline{h\cT_\infty}^*(F,G) \to \bR\cup \{ +\infty\}$. We also write this as $c(\mathchar`-; F,G)$. 
    By definition, $\alpha \in \overline{h\cT_\infty}^*(F,G)$ is zero if and only if $c(\alpha ;F,G)=+\infty$. 
    
    In some reasonable situations, $\overline{h\cT_\infty}^*(F,G)$ is isomorphic to $h\cT_\infty^*(F,G)$ (see \cref{lemma:specinfty}).

    In fact, $e$ can be taken as a non-zero idempotent in $\overline{h\cT_\infty}(T^*M)$. 
    However the author do not know an example of idempotents in $\overline{h\cT_\infty}(T^*M)$ which does not lift to that of $h\cT_\infty(T^*M)$.
\end{remark}

\subsection{Criteria for $\zeta_{\mathrm{MVZ}}$-superheaviness}

We also give a sufficient condition for $\zeta_{\mathrm{MVZ}}$-superheaviness.

\begin{theorem}\label{theorem:superheavy}
    Let $G \in \cT(T^*M)$.
    Assume that $\RS(G)$ is compact.
    \begin{enumerate}
        \item We assume that there exists $\beta \in h\cT_\infty^*(\bfk_{M\times [0,\infty)}, G) $ such that $c(\beta \mu_M;\bfk_{M\times [0,\infty)},  G\otimes \mathrm{or}_{M\times \bR})\neq +\infty$. Then $\RS(G)$ is $\zeta_{\mathrm{MVZ}}$-superheavy.
        \item We assume that there exists $\beta \in h\cT_\infty^*(G,\bfk_{M\times [0,\infty)}) $ such that $c(\mu_M\beta; G, \bfk_{M\times [0,\infty)}\otimes \mathrm{or}_{M\times \bR})\neq +\infty $. Then $\RS(G)$ is $\zeta_{\mathrm{MVZ}}$-superheavy.
    \end{enumerate}
\end{theorem}

\begin{proof}
    We only give a proof for (i) since the other is parallel. 

    By \cref{lemma:heavyequiv}, it is enough to prove $\zeta_e(H)\le c $ for $H \in C_{cc}(T^*M)$ with $H|_{U}\equiv c$ where $U$ is a neighborhood of $\RS(G)$. 

    Take $k\in \bZ$ and set $\ell \coloneqq \ell_+(kH)=c(\mu_M;\bfk_{M \times [0,\infty)},\cK_H^{\ostar k}(\bfk_{M \times [0,\infty)})\otimes \mathrm{or}_{M\times \bR})$ (see \cref{remark:ell+}).
    For $\varepsilon>0$, we take a lift
    \begin{equation}
        \tl{\mu}_M \colon \bfk_{M \times [0,\infty)} \to T_{-\ell+\varepsilon} \cK_H^{\ostar k}(\bfk_{M \times [0,\infty)})\otimes \mathrm{or}_{M\times \bR}[n]
    \end{equation}
    of $\mu_M$.
    For any $a < c(\beta; \bfk_{M \times [0,\infty)},G)$, we can also take a lift 
    \begin{equation}
        \tl{\beta} \colon \bfk_{M \times [0,\infty)} \to T_{-a} G[d]
    \end{equation}
    of $\beta$.
    By composing the above morphisms and noticing that $\cK_H^{\ostar k} (G) \simeq T_{kc} G$ by \cref{lemma:constantshift}, we get a morphism
    \begin{equation}
        \bfk_{M \times [0,\infty)} \xrightarrow{\tl{\mu}_M} T_{-\ell+\varepsilon}\cK_H^{\ostar k}(\bfk_{M \times [0,\infty)})\otimes \mathrm{or}_{M\times \bR}[n] \xrightarrow{\cK_H^{\ostar k} (\tilde{\beta})\otimes \id } T_{-\ell-a+\varepsilon + kc} G\otimes \mathrm{or}_{M\times \bR}[n+d],
    \end{equation}
    which is a lift of $\beta \mu_M \in Q^*_\infty(\bfk_{M \times [0,\infty)}, G)$.
    This means that 
    \begin{equation}
        \ell+a-\varepsilon - kc \le c(\beta \mu_M;\bfk_{M \times [0,\infty)}, G\otimes \mathrm{or}_{M\times \bR}).
    \end{equation}
    Since $\varepsilon >0$ and $a < c(\beta;\bfk_{M \times [0,\infty)}, G)$ are arbitrary, we obtain 
    \begin{equation}
        \ell_+(kH) = \ell \le kc + c(\beta \mu_M;\bfk_{M \times [0,\infty)}, G\otimes \mathrm{or}_{M\times \bR})-c(\beta;\bfk_{M \times [0,\infty)}, G).
    \end{equation}
    Taking the limit $k\to +\infty$ completes the proof.
\end{proof}

\subsection{$\zeta_{\mathrm{MVZ}}$-heavy/superheaviness subsets and Viterbo's conjecture}\label{subsec:Viterbo_conj}

In this subsection, we briefly review the relationship between heavy/superheavy subsets and Viterbo's spectral bound conjecture.

We define $DT^*M\coloneqq \{(x,\xi)\in T^*M \mid \|\xi\|_g \le 1\}$ for a Riemannian manifold $(M,g)$. 
The following was first conjectured by Viterbo in an earlier verion of \cite{viterbo2022symplectic}, as the special case $M=T^n$.
In this paper, we call this conjecture the Viterbo conjecture. 
Let $0_{M}$ be the image of the zero section of $T^*M$. 
\begin{conjecture}[Viterbo conjecture]
Let $(M,g)$ be a closed Riemannian manifold. There exists a constant $R>0$ such that, if $\phi \in \Ham_c(T^*M)$ satisfies $\phi (0_M)\subset DT^*M$, then $\gamma (\phi(0_M))<R$ holds.
\end{conjecture}

Note that, by \cite{Vi19}, we have 
\begin{equation}\label{eq:gammaLag}
    \gamma (\phi(0_M))= c(\mu_M; \bfk_{M\times[0,\infty)}, \cK_{\phi}^{\ostar}( \bfk_{M\times [0,\infty)})\otimes \mathrm{or}_{M\times\bR})-c(1; \bfk_{M\times[0,\infty)}, \cK_{\phi}^{\ostar}( \bfk_{M\times [0,\infty)}))
\end{equation}
where $\mu_M\in H^n(M;\mathrm{or}_M)$ is the fundamental class. 
The validity of this conjecture does not depend on the choice of the metric $g$ on $M$. It may depend on the coefficient field $\bfk$.

This conjecture is related to our discussion by the following theorem:

\begin{theorem}\label{theorem:Viterbotrue}
If $M$ satisfies the Viterbo conjecture, then any $\zeta_{\mathrm{MVZ}}$-heavy subset $A\subset T^*M$ is $\zeta_{\mathrm{MVZ}}$-superheavy. 
\end{theorem}
This statement is well-known to experts. 
See, for example, \cite[Rem.~1.23]{EP09rigid}, \cite[Thm.~4.1]{Entov14},  \cite[Prop.~5.1.2]{PR14}, \cite[Sec.~1.1.2]{shelukhin2022viterbo} for related statements. 
In this paper, we will see a slightly stronger statement \cref{proposition:Viterbocsheavy}. 

The following is an obvious corollary of \cref{corollary:heavy,theorem:Viterbotrue}.
\begin{corollary}\label{corollary:superheavyV}
    Let $G \in \cT(T^*M)$. Assume that $\RS(G)$ is compact and that the Viterbo conjecture for $M$ over $\bfk$ is true. 
    \begin{enumerate}
        \item If $\Spec(\bfk_{M\times [0,\infty)}, G) \neq \emptyset$,
    then $\RS(G)$ is $\zeta_{\mathrm{MVZ}}$-superheavy.
        \item If $\Spec(G,\bfk_{M\times[0,\infty)}) \neq \emptyset$,
    then $\RS(G)$ is $\zeta_{\mathrm{MVZ}}$-superheavy.
    \end{enumerate}
\end{corollary}

To the best of the author's knowledge, no counterexamples have been found, and the conjecture has been proven for some manifolds, which we will explain below.
We say that a $\bfk$-orientation is an isomorphism $\bfk_M\simeq \mathrm{or}_M$.
We call a manifold $M$ equipped with a $\bfk$-orientation a $\bfk$-oriented manifold. 
Note that the element $\mu_M$ can be regarded as an element of $H^n(M)=H^n(M;\bfk_M)$ if $M$ is $\bfk$-oriented.

First, Shelukhin~\cite{shelukhin2022viterbo} proved the conjecture for $\bfk=\bZ/2\bZ$ and $M=S^n, \bR P^n, \bC P^n, \allowbreak \bH P^n$.
Next, Shelukhin~\cite{shelukhin2022symplectic} provided a proof for the case where $M$ is $\bfk$-oriented and string point-invertible. The condition of string point-invertibility is a condition on the Gerstenhaber algebra structure on the homology of the free loop space of $M$.
Also, with a minor assumption on the coefficient field, Guillermou--Vichery~\cite{guillermou2022viterbos} proved the conjecture for homogeneous spaces of compact Lie groups. 
At around the same time, Viterbo~\cite{viterbo2022inverse} proved the conjecture for a mysterious class that includes the compact Lie groups, and provided an alternative proof for homogeneous spaces without the assumption of the coefficient field.
To be precise, the following holds:

\begin{theorem}[{\cite{shelukhin2022viterbo,shelukhin2022symplectic,guillermou2022viterbos,viterbo2022inverse}}]
\begin{enumerate}
\item For a compact Lie group $\mathsf{G}$ and its closed subgroup $\mathsf{H}$, the homogeneous space $\mathsf{G}/\mathsf{H}$ satisfies the Viterbo conjecture.
\item If $M$ is $\bfk$-oriented and string point invertible, then it satisfies the Viterbo conjecture.
\item If there exist compact manifolds $V$, $P$ and morphisms $f\colon P\to M, g\colon V\to \mathrm{Diff} (P)$ such that $(f \circ \mathrm{ev}_y\circ g )^*(\mu_M)\neq 0 $ where $\mathrm{ev}_y$ is the evaluation at some point $y\in P$, then $M$ satisfies the Viterbo conjecture.
\item If there exists a morphism $p\colon N\to M$ such that $p^*(\mu_M)\neq 0$, and $N$ satisfies the Viterbo conjecture, then $M$ also satisfies it.
\end{enumerate}
\end{theorem}

\section{A characterization of heaviness}\label{section:characterization}

Let $e\colon F\to F$ be an idempotent in $h\cT_\infty(T^*M)$ with $c(e;F,F)\neq +\infty$ and denote the associated partial symplectic quasi-state by $\zeta_e$. 
To prove a given compact subset $A\subset T^*M$ is $\zeta_e$-heavy using \cref{theorem:zetae_heavy}, one need to look for $G\in \cT_A(T^*M)$ satisfying a condition in the theorem. 

In fact, there is a most sensitive pair of an object $G$ with $\RS(G)\subset A$ and a morphism $\beta\colon F\to G$ (resp. $\beta\colon G\to F$) with respect to the condition in (i) (resp. (ii)) of \cref{theorem:zetae_heavy}. 
These object and morphism are given by the left (resp. right) adjoint $\iota_A^*$(resp. $\iota_A^!$) of the inclusion functor ${\iota_A}_* \colon \cT_A(T^*M)\to \cT(T^*M)$ and unit (resp.~counit) morphism of the adjunction. An explicit description of these adjoint functors are given by Kuo~\cite{Kuowrap}. 
Since ${\iota_A}_*$ is the natural inclusion, it will be omitted in many cases below.

Let us consider the following conditions for a compact subset $A\subset T^*M$ and an idempotent $e\colon F\to F$ in $h\cT_\infty(T^*M)$ with $c(e;F,F)\neq +\infty$:
\begin{enumerate}
    \item $A$ is $\zeta_e$-heavy,
    \item there exist an object $G\in \cT_A(T^*M)$ and an element $\beta \in h\cT_\infty^* (F,G)$ satisfying $c(\beta e;F,G)\neq +\infty$,
    \item the morphism $\eta_{A,F} \colon F\to \iota_A^*F$ in $h\cT_\infty(T^*M)$ induced by the unit morphism of the adjunction satisfies  $c(\eta_{A,F} e; F, \iota_A^*F)\neq +\infty$,
    \item there exist an object $G\in \cT_A(T^*M)$ and an element $\beta \in h\cT_\infty^*(G,F)$ satisfying $c( e \beta ; G,F)\neq +\infty$, 
    \item the morphism $\epsilon_{A,F} \colon  \iota_A^!F\to F$ in $h\cT_\infty(T^*M)$ induced by the counit morphism of the adjunction satisfies  $c( e \epsilon_{A,F} ; \iota_A^!F, F)\neq +\infty$. 
\end{enumerate}

The implications (iii) $\Rightarrow$ (ii) and (v) $\Rightarrow$ (iv) are trivial. 
The implications (ii) $\Rightarrow$ (i) and (iv) $\Rightarrow$ (i) are \cref{theorem:zetae_heavy}. 
The implications (ii) $\Rightarrow$ (iii) and (iv) $\Rightarrow$ (v) follow from formal properties of the adjunctions and the spectral invariants. 

Under some assumption on $F$, we will show that the conditions (i), (ii), (iii) are equivalent. 
See \cref{remark:counit} for the dual statements (iv), (v). 

\begin{theorem}\label{theorem:characterization}
    Assume that $F$ is cohomologically constructible and $q_{\bR} \pi (\MS(F)\cap \{\tau >0\})$ is bounded in $\bR$.
    Then the following conditions for a compact subset $A\subset T^*M$ and an idempotent $e\colon F\to F$ are equivalent:  
    \begin{enumerate}
    \item $A$ is $\zeta_e$-heavy,
    \item there exist an object $G\in \cT_A(T^*M)$ and an element $\beta \in h\cT_\infty^* (F,G)$ satisfying $c(\beta e; F,G)\neq +\infty$,
    \item[\textup{(ii')}] there exist an object $G\in \cT_A(T^*M)$ and an element $\beta \in h\cT_\infty^* (F,G)$ satisfying $\beta e\neq 0 \in h\cT_\infty^* (F,G)$,
    \item the morphism $\eta_{A,F} \colon F\to  \iota_A^*F$ in $h\cT_\infty(T^*M)$ induced by the unit morphism of the adjunction satisfies $c(\eta_{A,F} e;F, \iota_A^*F)\neq \infty$. 
    \item[\textup{(iii')}] the morphism $\eta_{A,F} \colon F\to  \iota_A^*F$ in $h\cT_\infty(T^*M)$ induced by the unit morphism of the adjunction satisfies $\eta_{A,F} e\neq 0\in h\cT_\infty^*(F, \iota_A^*F)$, 
\end{enumerate}

\end{theorem}

There are many $F$ satisfying the assumption of the proposition. 
However, this assumption imposes a non-trivial constraint on the partial symplectic quasi state $\zeta_e$. See \cref{proposition:psqsconstraint} below. 

The main part of the proof is to prove (i) $\Rightarrow$ (iii).  
The following proposition is a key observation for this part. 

\begin{proposition}\label{proposition:spectral-colim}
    Let $F$ be an object of $\cT(T^*M)$ and $G_\bullet \colon \cN \bN\to \cT(T^*M)$ be a functor.
    Assume the following:
    \begin{enumerate}
    \renewcommand{\labelenumi}{$\mathrm{(\arabic{enumi})}$}
        \item $F$ is cohomologically constructible and $q_{\bR} \pi (\MS(F)\cap \{\tau >0\})$ is bounded in $\bR$.
        \item $q_{\bR} \pi (\MS(G_i) \cap \{\tau >0\})$ are uniformly bounded below. 
    \end{enumerate}
    Let $(\beta_i\colon F\to G_i)_{i\in \bN}$ be morphisms which are compatible with structure morphisms of $G_\bullet$ and $\beta_\infty\colon F\to \colim G_\bullet$ be the induced morphism. 
    Then 
    \begin{equation}
        \lim_i c(\beta_i;F,G_i)=c(\beta_\infty;F,\colim G_\bullet). 
    \end{equation}
\end{proposition}

\begin{proof}
    $\lim_i c(\beta_i;F,G_i)\le c(\beta_\infty;F,\colim_i G_i)$ is obvious. 
    If the conclusion is not true, we can take $a,b\in \bR$ satisfying $\lim_i c(\beta_i;F,G_i)<a<b<c(\beta_\infty;F,\colim_i G_i)$.

    For any $R'\ge 0, i\in \bN$ and $\varepsilon \in [0,b-a]$, 
    the morphism $F\star \bfk_{M\times [-R',\infty)}\to G_i$ induced by $\beta_i$ can not be lifted to a morphism $F\star \bfk_{M\times [a+\varepsilon,\infty)}\to G_i$.
    This means that the induced morphism $F\star \bfk_{M\times [-R',a+\varepsilon)}\to G_i$ is non-zero. 
    On the other hand, there exists a real number $R>0$ such that
    $F\star \bfk_{M\times [-R+\varepsilon , a+\varepsilon)}\to \colim_i G_i$ is zero for any $\varepsilon\in [0,b-a]$.  

    By the adjunctions and \cref{lemma:TamarkinVerdier}, for any $G\in \cT(T^*M)$, 
    \begin{equation}
        \Hom_{\cT(T^*M)}(F\star \bfk_{M\times [-R+\varepsilon, a+\varepsilon)},G)\simeq \Hom_{\cT (\mathrm{pt})}(\bfk_{[\varepsilon,\infty)}, q_{\bR *}(D^T(F\star \bfk_{M\times [-R,a)})\star G))
    \end{equation}
    holds. 

    Since $q_{\bR *}\simeq q_{\bR !}$, 
    \begin{equation}
        \colim_i q_{\bR *}(D^T(F\star \bfk_{M\times [-R,a)})\star G_i) \simeq q_{\bR *}(D^T(F\star \bfk_{M\times [-R,a)})\star \colim_i G_i ). 
    \end{equation}

    Put $H_\bullet \colon \cN\bN\to \cT (\mathrm{pt})$ be the functor $q_{\bR *}(D^T(F\star \bfk_{M\times [-R,a)})\star G_\bullet)[-1]$ and $H_\infty \coloneqq \colim H_\bullet$. 
    Since $\bfk_{[\varepsilon, \infty)}\simeq \bfk_{(-\infty, \varepsilon)}[1]$ in $\cT(\mathrm{pt})$, 
    for any $\varepsilon \in [0,b-a]$ and $i\in \bN$, the induced morphism $\bfk_{(-\infty, \varepsilon)}\to  H_i$ is non-zero. 
    On the other hand, the induced morphism 
    $\bfk_{(-\infty, \varepsilon)}\to H_\infty$ is zero for each $\varepsilon \in [0,b-a]$.

    For each $H\in \Sh_{\{\tau \ge 0 \}}(\bR)$ and $c\in \bR$,  
    the stalk $H_c$ is isomorphic to $\colim_{\varepsilon'\to +0}\Gamma((-\infty , c+\varepsilon');H)$. 
    Let us identify $\cT (\pt)$ with the subcategory of $\Sh_{\{\tau \ge 0\}} (\bR)$. 
    Hence the morphisms $(\bfk_{(-\infty, \varepsilon)}\to H_i)_{\varepsilon \in [0,b-a]}$ above induce a nonzero element in the cohomology $H^0((H_i)_\varepsilon)$ of the stalk for each $\varepsilon \in [0,b-a)$ and $i\in \bN$. 
    
    Since stalk and cohomology commute with colimits, 
    $\colim_i H^0((H_i)_\varepsilon)\simeq H^0((H_\infty)_\varepsilon)$ and hence a non-zero element in $H^0((H_\infty)_\varepsilon)$ is induced for each $\varepsilon \in [0,b-a)$. 
    This contradicts to that 
    $\bfk_{(-\infty, \varepsilon')}\to H_\infty$ is zero for each $\varepsilon' \in (\varepsilon,b-a]$. 
\end{proof}

\begin{corollary}\label{corollary:unitinfty}
    Assume that $F$ is cohomologically constructible and $q_{\bR} \pi (\MS(F)\cap \{\tau >0\})$ is bounded in $\bR$.
    Then the following are equivalent:
    \begin{enumerate}
        \item $c(\beta e;F,\cK_A \ostar F)=+\infty$,
        \item $c(\beta e;F,\cK_A \ostar F)>0$,
        \item for any $s\in \bR$ there exists $H^s\in C_{cc}(T^*M, A)$ with $c(e;F,\cK_{H^s} \ostar F)>s$,
        \item there exists $H\in C_{cc}(T^*M, A)$ with $c(e;F,\cK_H \ostar F)>0$.
    \end{enumerate}
\end{corollary}
\begin{proof}
    (iii) $\Rightarrow$ (iv) is trivial. 
    
    (iv) $\Rightarrow$ (ii) is obvious by the inequality $c(e;F,\cK_H \ostar F)\le c(\beta e;F,\cK_A \ostar F)$ for any $H\in C_{cc}(T^*M, A)$. 

    (ii) $\Rightarrow$ (i) is deduced by \cref{equation:kernelidempotnce} $\cK_A \ostar \cK_A\simeq \cK_A$ . 
    If $c(\beta e;F,\cK_A \ostar F)>0$, then for each $s\in (0, c(\beta e;F,\cK_A \ostar F))$ there exists a lift $\tl{\alpha}\colon F\to T_{-s} \cK_A \ostar F$ of $\beta e$. 
    Then the composition 
    \begin{equation}
        T_{-s}\cK_A^{\ostar}(\tl{\alpha})\circ\tl{\alpha}\colon F\to T_{-s} \cK_A \ostar F\to T_{-2s} \cK_A \ostar \cK_A \ostar F\simeq  T_{-2s} \cK_A \ostar F
    \end{equation}
    is also a lift of $\beta e$ and hence $2s\le c(\beta e;F,\cK_A \ostar F)$. This means $c(\beta e;F,\cK_A \ostar F)=+\infty$.

    (i) $\Rightarrow$ (iii) follows from \cref{proposition:spectral-colim} as follows. 
    Let us prove the contraposition $\neg$(iii) $\Rightarrow$ $\neg$(i) and assume that there exists $s\in \bR$ such that $c(e;F,\cK_{H} \ostar F)\le s$ for any $H\in C_{cc}(T^*M, A)$. 
    By fixing a final functor $H_\bullet\colon \cN \bN \to \cC_{cc}(T^*M, A)$, we can write $\cK_A\ostar F$ as the sequential colimit $\colim_n \cK_{H_n}\ostar F$. 
    By the assumption, $c(e;F,\cK_{H_n} \ostar F)\le s$ holds for each $n\in \bN$. By \cref{proposition:spectral-colim}, $c(\beta e;F,\cK_A \ostar F)=\lim_n c(e;F,\cK_{H_n} \ostar F)\le s <+\infty$ is deduced. Note that the assumption (2) in the \cref{proposition:spectral-colim} is guaranteed by the assumption on $F$. 
\end{proof}

\begin{proof}[Proof of \cref{theorem:characterization}.]
    ``(iii)$\Rightarrow$ (ii)'' and ``(ii) $\Rightarrow$ (ii')'' are trivial implications. 
    By \cref{lemma:specinfty}, the condition (iii) is equivalent to (iii'). 
    By the universal property of $ \iota_A^*F$, ``(ii') $\Rightarrow$ (iii')'' holds.
    Hence all the four conditions (ii), (ii'), (iii) and (iii') are equivalent. 
    
    \cref{theorem:zetae_heavy} asserts ``(ii) $\Rightarrow$ (i)''. 
    
    We prove the contraposition ``$\neg$(iii') $\Rightarrow$ $\neg$(i)''. 
    Let us assume $c(\beta e;F, \iota_A^*F)=+\infty$.
    Then there exists $H\in C_{cc}(T^*M, A)$ with $c(e;F,\cK_{H} \ostar F)>0$ by \cref{corollary:unitinfty}. 
    Hence $\zeta_e(-H)\le \ell_e (-H)=-c(e;F,\cK_{H} \ostar F)<0$, which means that $A$ is not $\zeta_e$-heavy.
\end{proof}

For the case $e=\id_F$, we obtain a slightly refined version of \cref{theorem:characterization} by the universal property of $ \iota_A^*F$.
\begin{corollary}\label{corollary:heavynonzero}
    Assume that $F$ is cohomologically constructible and $q_{\bR} \pi (\MS(F)\cap \{\tau >0\})$ is bounded in $\bR$.
    For any compact subset $A\subset T^*M$,
    $A$ is $\zeta_{F}$-heavy if and only if $ \iota_A^*F\not\simeq 0$ in $h\cT_\infty (T^*M)$. 
\end{corollary}
\begin{proof}
    If $A$ is $\zeta_{F}$-heavy, $ \iota_A^*F\not\simeq 0$ in $h\cT_\infty (T^*M)$ by (i) $\Rightarrow$ (iii') of \cref{theorem:characterization}. 

    If $ \iota_A^*F\not\simeq 0$ in $h\cT_\infty (T^*M)$, the natural morphism $\tau_{c}\colon  \iota_A^*F\to T_c \iota_A^*F$ is non-zero for any $c\ge 0$.
    By the adjunction isomorphism $\Hom_{\cT(T^*M)} (F,T_c \iota_A^*F)\simeq \Hom_{\cT_A(T^*M)} ( \iota_A^*F,T_c \iota_A^*F)$, the composite $\tau_c\circ \eta_{A,F}$ is also non-zero for any $c\ge 0$. 
    This means $\eta_{A,F}$ is non-zero as a morphism of $\cT_\infty(T^*M)$. 
    The assertion follows from (iii') $\Rightarrow$ (i) of \cref{theorem:characterization}. 
\end{proof}

\begin{remark}\label{remark:counit}
    The author do not know whether the dual statement (i) $\Rightarrow$ (v) holds or not. 
    At least, the dual statement of the key proposition \cref{proposition:spectral-colim} does not hold. 
    
    Let $V_0$ be the direct sum $\bfk^\bN$. For each $n\in \bN$, define $V_n\subset V_0$ as the subspace consisting of the elements $(a_i)_{i\in \bN}$ with $a_i=0$ for each $0\le i <n$.
    Consider the linear map $V_0\to \bfk$ as taking the sum of its components.
    Let $F_n\coloneqq (V_n)_{[0,\infty)}$ and $G\coloneqq \bfk_{[0,\infty)}$.
    Then we obtain a projective system $F_\bullet\colon \cN\bN^{\mathrm{op}}\to \cD (\mathrm{pt})$ and consistent morphisms $\beta_i\colon F_i\to G$. In this case, $c(\beta_i;F_i, G)=0$ for each $i\in \bN$, however $c(\beta_\infty; \lim F_\bullet, G)=+\infty$ since $\lim F_\bullet=0$.
\end{remark}

\begin{definition}
    A compact subset $A\subset T^*M$ is \emph{cohomologically superheavy} if the  morphism $\beta\mu_M\colon \bfk_{M\times [0,\infty)}\to  \iota_A^*\bfk_{M\times [0,\infty)}\otimes \mathrm{or}_{M\times \bR}[n]$ 
    is non-zero in $h\cT_\infty (T^*M)$. 
\end{definition}

By \cref{theorem:superheavy,lemma:specinfty}, cohomological superheaviness implies $\zeta_{\mathrm{MVZ}}$-superheaviness.
The author do not know whether the inverse holds in general. 
Heaviness/superheviness is preserved by Hausdorff limits, that is clear by \cref{lemma:heavyequiv}. 
On the other hand, the corresponding property for cohomological superheaviness is not obvious. 
However, the Viterbo conjecture implies the inverse statement. 

\begin{proposition}\label{proposition:Viterbocsheavy}
    If the Viterbo conjecture for $M$ is true over $\bfk$, every $\zeta_{\mathrm{MVZ}}$-heavy subset in $T^*M$ is cohomologically superheavy.
\end{proposition}
\begin{proof}
    Take $A\subset T^*M$ is a $\zeta_{\mathrm{MVZ}}$-heavy subset.
    We may assume $A$ is contained in the interior of $DT^*M$ by choosing the Riemannian metric on $M$ sufficiently small. 
    By (i) $\Leftrightarrow$ (iii) of \cref{theorem:characterization}, $c(\eta_{A,\bfk_{M\times[0,\infty)}}; \bfk_{M\times[0,\infty)}, \cK_A^{\ostar} (\bfk_{M\times [0,\infty)}))\neq \infty$ holds. 
    By (i) $\Leftrightarrow$ (iv) of \cref{corollary:unitinfty}, for any $H\in C_{cc}(T^*M,A)$, $c(1; \bfk_{M\times[0,\infty)}, \cK_H^{\ostar} (\bfk_{M\times [0,\infty)}))\le 0$. 
    Take a final functor $H_\bullet\colon \cN\bN\to \cC_{cc}(T^*M,A)$ satisfying $0\le H_n$ and $H_n$ is constant outside $DT^*M$ for each $n$.  
    By the monotonicity of $c$, $c(1; \bfk_{M\times[0,\infty)}, \cK_{H_n}^{\ostar }(\bfk_{M\times [0,\infty)}))=0$ for each $n$. 
    If the Viterbo conjecture for $M$ is true over $\bfk$, since $\phi_1^{H_n}(0_M)\subset DT^*M$, there exists a real number $R>0$ such that $\gamma(\phi_1^{H_n}(0_M))<R$ for each $n$. 
    Hence by \cref{eq:gammaLag}, we obtain
    \begin{equation}
        c(\mu_M; \bfk_{M\times[0,\infty)}, \cK_{H_n}^{\ostar} (\bfk_{M\times [0,\infty)})\otimes \mathrm{or}_{M\times\bR})<R
    \end{equation}
    for each $n$. 
    By \cref{proposition:spectral-colim}, we also get 
    \begin{equation}
        c(\eta_{A,\bfk_{M\times[0,\infty)}}\mu_M; \bfk_{M\times[0,\infty)}, \cK_{A}^{\ostar} (\bfk_{M\times [0,\infty)})\otimes \mathrm{or}_{M\times\bR})\le R. 
    \end{equation}
    This implies $A$ is cohomologically superheavy. 
\end{proof}

Let us conclude this section by observing that the assumption on $F$ imposes a constraint on $\zeta_e$. 
\begin{proposition}\label{proposition:psqsconstraint}
    Assume that $F$ is cohomologically constructible and $q_{\bR} \pi (\MS(F)\cap \{\tau >0\})$ is bounded in $\bR$.
    Let $e\colon F\to F$ be an idempotent in $\cT_\infty(T^*M)$ with $c(e;F,F)<+\infty$.
    There exists $G\in \Sh(M)$ such that
    $\zeta_e=\zeta_{G\boxtimes\bfk_{[0,\infty)}}$. 
    Moreover if $\RS(F)$ is compact, then $G$ is cohomologically locally constant.  
\end{proposition}
\begin{proof}
    Consider the morphism
    \begin{equation}
        r\colon M\times \bR\times \bR_{>0}\to M\times \bR; (x,t,s)\mapsto \left(x,\frac{t}{s}\right)
    \end{equation}
    and define $F_n\coloneqq r^{-1}(F)|_{M\times \bR\times \left\{\frac{1}{n}\right\}}$ for each $n\in \bZ_{\ge 1}$.
    By \cite{AI20}, $d(F_n, F_m)< +\infty$ for any $m,n \in \bZ_{\ge 1}$ and the sequence $(F_n)_n$ is Cauchy. 
    Let $F_\infty$ be a convergence point of $(F_n)_n$, which exists by \cite{AI22completeness} or \cite{guillermou2022gamma}. 
    Let $e'\colon F_\infty \to F_\infty$ be an idempotent corresponding to the idempotent $e\colon F\to F$ through some isomorphism $F\simeq F_\infty$ in $h\cT_{\infty}(T^*M)$ by \cref{remark:finitelength}.

    By a microsupport estimation, we can see that $F_\infty$ is of the form $G'\boxtimes \bfk_{[0,\infty)}$ with $G'\in \Sh(M)$. 
    If $\RS(F)$ is compact, the microsupport estimate also says $G'$ is cohomologically locally constant. 
    The idempotent in $e'$ corresponds to an idempotent $e''\colon G\to G$ in $h\Sh(M)$. 
    Since $h\Sh(M)$ is idempotent complete by \cite{bokstedt1993homotopy}, $e''$ corresponds to some object $G$ which is a retract of $G'$. 
    Note that if $G'$ is cohomologically locally constant, so is the retract $G$.  
\end{proof}

\section{Examples}\label{section:examples}

In this section, we give some examples that we can check the $\zeta_{\mathrm{MVZ}}$superheaviness by our theorems and known results about the Viterbo conjecture.

\begin{example}\label{example:contifunc}
    For a continuous function $f\colon M\to \bR$, we define a subset $Z_f\subset M\times \bR$ as $\{(x,t) \in M\times \bR \mid t\ge -f(t) \}$ and define $F_f\coloneqq \bfk_{Z_f}$.
    If $f$ is of class $C^1$, the derivative $df\colon M\to T^*M$ is defined as a continuous section. In this case, we denote by $\Gamma_{df}\subset T^*M$ the image of $df$. We then have that $\RS (F_f)=\Gamma_{df}$.
    For a general continuous function $f$, 
    $\RS (F_f)$ is called the Vichery subdifferential \cite{VicheryHDC}. 
    Note that $\RS (F_f)$ is not necessarily compact for a general continuous function. However, if $f$ is Lipschitz continuous, then $\RS(F_f)$ is compact.

    If $\RS(F_f)$ is compact, it is $\zeta_{\mathrm{MVZ}}$-superheavy by \cref{theorem:superheavy}.
\end{example}

\begin{example}\label{example:exactLag}
    Any compact exact Lagrangian submanifold in a cotangent bundle is $\zeta_{\mathrm{MVZ}}$-superheavy. This is deduced from \cref{theorem:superheavy} and the existence result \cite{Gu12,Gu23cotangent,Vi19} of the sheaf quatizations of compact exact Lagrangian submanifolds.
\end{example}

The completion of the space of Lagrangians with respect to the spectral metric gives interesting examples. 
As a common generalization of \cref{example:contifunc,example:exactLag}, we have the following. 

\begin{example}
    Viterbo~\cite{viterbo2022supports} introduced the notion of $\gamma$-supports for elements of the completion of the space of some Lagrangian submanifolds with respect to the spectral distance.
    The $\gamma$-support of an element of the completion coincides with the reduced microsupport of an associated sheaf quantization \cite{AGHIV}. 
    Especially, a $\gamma$-support is realized as a reduced microsupport of a sheaf $F$ with $d(\bfk_{M\times [0,\infty)}, F)<\infty$, 
    Hence, compact $\gamma$-supports are $\zeta_{\mathrm{MVZ}}$-superheavy by \cref{theorem:superheavy}.
\end{example}

\begin{remark}
    A Birkhoff attractor is a closed subset of $T^*S^1$ associated to a map called conformally exact symplectic map on $T^*S^1$. 
    It can be an indecomposable continuum. 
    See \cite{AHV24} and its references.
    The Birkhoff attractors are realized as $\gamma$-supports by \cite[Thm.~1.3]{AHV24}. 
    Hence we obtain an examples of a $\zeta_{\mathrm{MVZ}}$-superheavy subset which is an indecomposable continuum in $T^*S^1$.
\end{remark}

With a morphism of manifolds $f \colon X \to Y$, we associate the following commutative diagram of morphisms of manifolds:
\begin{equation}\label{diag:fpifd}
\begin{aligned}
\xymatrix{
	T^*X \ar[d]_-{\pi} & X \times_Y T^*Y \ar[d] \ar[l]_-{f_d}
	\ar[r]^-{f_\pi} & T^*Y \ar[d]^-{\pi} \\
	X \ar@{=}[r] & X \ar[r]_-f & Y,
}
\end{aligned}
\end{equation}
where $f_\pi$ is the projection and $f_d$ is induced by the transpose
of the tangent map $f' \colon TX \to X \times_Y TY$.

The following lemma is a consequence of \cref{corollary:heavy}. 
\begin{lemma}\label{lemma:heavydescent}
    Let $p \colon N \to M$ be a submersion.
    Let $G \in \cT(T^*N)$ and assume that $\Spec(\bfk_{M\times [0,\infty)}, G) \neq \emptyset$ and $\tilde{p}=p \times \id_\bR$ is proper on $\Supp(G)$ and moreover $\RS(G)\cap \Image (p_d)$ is compact.
    Then $\RS(\tilde{p}_*G)$ is $\zeta_{\mathrm{MVZ}}$-heavy.
\end{lemma}

Note that under the situation of the above lemma, we have $\RS(\tilde{p}_*G) \subset p_\pi p_d^{-1}\RS(G)$.

\begin{example}\label{example:realproj}
    Take $S^{n}=\{(x_0,\dots, x_{n+1})\in \bR^{n+1} \mid \sum_{i=0}^{n+1}x_i^2=1 \}$, $ \bR P^n\coloneqq S^n /(\bZ/2\bZ)$ and let $p\colon S^n\to \bR P^n$ be the quotient map. 
    Let us take $f\colon S^{n}\to \bR; (x_0,\dots ,x_{n+1})\mapsto x_0$. 
    Then $\iota \coloneqq p_\pi \circ(p_d)^{-1}\circ df\colon S^n\to T^*\bR P^n$ is an exact Lagrangian immersion. 

    Applying \cref{lemma:heavydescent} to $p\colon S^n\to \bR P^n$ and $F_f=\bfk_{Z_f}\in \cT(T^*S^n)$, we obtain $\iota (S^n)\subset T^*\bR P^n$ is $\zeta_{\mathrm{MVZ}}$-heavy. 
    Since the Viterbo conjecture for $\bR P^n$ is true over any coefficient field, $\iota (S^n)\subset T^*\bR P^n$ is $\zeta_{\mathrm{MVZ}}$-superheavy by \cref{corollary:superheavyV}. 
    If the characteristic of $\bfk$ is not $2$, the superheaviness can be deduced directly from \cref{theorem:superheavy}.

    For $n=1$, $\iota (S^1)\subset T^*S^1$ is the singular fiber of the pendulum. 
\end{example}

\begin{example}\label{example:spherical}
    Let us take $S^{2n+1}=\{(z_0,\dots, z_{n+1})\in \bC^{n+1} \mid \sum_{i=0}^{n+1}|z_i|^2=1 \}$ and consider the quotient map $p\colon S^{2n+1}\to \bC P^n$ by the $\mathrm{U}(1)$-action. 
    
    Let us take $f\colon S^{2n+1}\to \bR; (z_0,\dots ,z_{n+1})\mapsto \Re z_0$ and then the fiberwise singular points of $f$ along $p$ is the sphere $S^{2n}=\{(z_0,\dots, z_{n+1})\in S^{2n+1}\mid z_0\in \bR \}$. 
    Then the differential of $f$ on $S^{2n}$ along the base direction gives an exact Lagrangian immersion $\iota\colon S^{2n}\to T^*\bC P^n$.  

    Applying \cref{lemma:heavydescent} to $p\colon S^{2n+1}\to \bC P^n$ and $F_f=\bfk_{Z_f}\in \cT(T^*S^{2n+1})$, we obtain $\iota (S^{2n})\subset T^*\bC P^n$ is $\zeta_{\mathrm{MVZ}}$-heavy. 
    Since the Viterbo conjecture for $\bC P^n$ is true over any coefficient field, $\iota (S^{2n})\subset T^*\bC P^n$ is $\zeta_{\mathrm{MVZ}}$-superheavy.

    For $n=1$, $\iota (\Sigma_f)\subset T^* S^2$ is the superheavy singular fiber of the spherical pendulum, whose superheaviness is proved by \cite{KO22rigidfiber}. 
    See \cite{CB15} for detailed description for the fibers of the spherical pendulum. 
    
\end{example}
\begin{remark}\label{remark:quatproj}
     \cref{example:realproj,example:spherical} can be extended to the quaternionic version.\footnote{Lagrangian immersions in \cref{example:spherical,remark:quatproj} were treated in an intensive lecture series by Kenji Fukaya at Kyoto University in January 2023.}
\end{remark}

\begin{example}[The Lagrange top] 
    Let be $q\colon\mathrm{SO}(3)\to S^2$ the map picking up the first column of each matrix. 
    In $n=1$ case of \cref{example:spherical}, the superheaviness of the singular fiber of the spherical pendulum is proved by the existence of the object $p_*F_f\in \cT(T^*S^2)$. 
    The reduced microsupport of $q^{-1}p_*F_f\in \cT(T^*\mathrm{SO}(3))$ is a singular fiber of the Lagrange top. 
    By \cref{corollary:superheavyV}, the singular fiber is $\zeta_{\mathrm{MVZ}}$-superheavy. 
    This is also proved by \cite{KO22rigidfiber} and
    see \cite{CB15} also for detailed description for the fibers of the Lagrange top. 
\end{example}

\appendix

\section{Partial symplectic quasi-states from kernels}\label{sec:psqsHam}

In this appendix, we discuss partial symplectic quasi-states that are directly constructed from GKS kernels. 
Many of the proofs are parallel to those handled in the main text.

We define $\zeta_{\mathrm{Ham}}\colon C_{cc}(T^*M)\to \bR$ by
\begin{equation}
    \zeta_{\mathrm{Ham}}(H) \coloneqq -\lim_k\frac{c(1;\bfk_{\Delta_M\times [0,\infty)}, \cK_{H}^{\ostar -k})}{k}.
\end{equation}
This functional will be a partial symplectic quasi-states and correspond to one defined in \cite{Lanzat13}.
More generally, for $K \in \cT(T^*M\times T^*M)$ an idempotent $e\colon K \to K$ in $\cT_\infty (T^*M\times T^*M)$ with $c(e;K,K)=+\infty$,
we can define
\begin{equation}
    \zeta_{\mathrm{Ham}, e}(H) \coloneqq -\lim_k\frac{c(e;K, \cK_{H}^{\ostar -k}(K))}{k}.
\end{equation}
\begin{proposition}
    For each $K \in \cT(T^*M\times T^*M)$ and each idempotent $e\colon K \to K$ in $\cT_\infty (T^*M\times T^*M)$ with $c(e;K,K)=+\infty$, 
     $\zeta_{\mathrm{Ham}, e}$ is a partial symplectic quasi-state. 
\end{proposition}
\begin{proof}
    The proof is parallel to that of \cref{theorem:idempotent_quasistate} by replacing $\ell_e (\phi)\coloneqq -c(e;\cK_{\phi}^{\ostar} (F), F)$ by $-c(e;\cK_{\phi}^{\ostar} (K), K)$.
\end{proof}
\begin{proposition}
    \begin{enumerate}
        \item For each $F \in \cT(T^*M)$, each idempotent $e\colon F \to F$ in $\cT_\infty (T^*M)$ with $c(e;F,F)=+\infty$, and each $H\in C_{cc}(T^*M)$, $\zeta_e (H)\le \zeta_{\mathrm{Ham}}(H)$.
        \item For each $K \in \cT(T^*M\times T^*M)$, each idempotent $e\colon K \to K$ in $\cT_\infty (T^*M\times T^*M)$ with $c(e;K,K)=+\infty$, and each $H\in C_{cc}(T^*M)$,
        $\zeta_{\mathrm{Ham}, e} (H)\le \zeta_{\mathrm{Ham}}(H)$.
    \end{enumerate}
\end{proposition}
\begin{proof}
    (i) For each $k$,  
    \begin{equation}
        c(e;F,\cK_H^{\ostar -k}(F))\ge c(1;\bfk_{\Delta_M\times [0,\infty)}, \cK_{H}^{\ostar -k}) +c(e;F,F)
    \end{equation}
    holds. Hence by taking the limit $k\to \infty$, we obtain the assertion. 

    (ii) is parallel to (i). 
\end{proof}
\begin{corollary}
The following hold:
\begin{enumerate}
    \item Every $\zeta_e$-heavy subset is $\zeta_{\mathrm{Ham}}$-heavy for each $F \in \cT(T^*M)$ and each idempotent $e\colon F \to F$ in $\cT_\infty (T^*M)$ with $c(e;F,F)=+\infty$.
    \item Every $\zeta_{\mathrm{Ham}, e}$-heavy subset is $\zeta_{\mathrm{Ham}}$-heavy  for each $K \in \cT(T^*M\times T^*M)$ and each idempotent $e\colon K \to K$ in $\cT_\infty (T^*M\times T^*M)$ with $c(e;K,K)=+\infty$.
\end{enumerate}
\end{corollary}

The following proposition is a direct consequence of \cref{lemma:cnonposi}.
\begin{proposition}
    $\zeta_{\mathrm{Ham}}(H)\ge 0 $ for each $H\in C_c(T^*M)$.
\end{proposition}

\begin{corollary}
    There is no $\zeta_{\mathrm{Ham}}$-superheavy subset. 
\end{corollary}

For a compact subset $A$ of $T^*M$, we defined
\begin{equation}
    \cK_A\coloneqq \colim_{H\in \cC_{cc}(T^*M,A)} \cK_H. 
\end{equation}

\begin{theorem}\label{theorem:Hamchar}
    For a compact subset $A$ of $T^*M$, the following are equivalent:
    \begin{enumerate}
        \item $A$ is $\zeta_{\mathrm{Ham}}$-heavy,
        \item the canonical morphism $\bfk_{\Delta\times [0,\infty)}\to \cK_A$ is non-zero as a morphism of $h\cT_\infty(T^*M\times T^*M)$,  
        \item $\cK_A \not\simeq 0$ in $h\cT_\infty(T^*M\times T^*M)$. 
    \end{enumerate}
\end{theorem}
\begin{proof}
    A proof of (i) $\Rightarrow$ (ii) is parallel to that of (i) $\Rightarrow$ (iii') of \cref{theorem:characterization}.

    A proof of (ii) $\Rightarrow$ (i) is parallel to that of (i) of \cref{theorem:zetae_heavy}, once one replaced \cref{lemma:constantshift} by \cref{lemma:Kernelconst} below.

    The implication (ii) $\Rightarrow$ (iii) is trivial. 

    Applying $\cK_A^{\ostar}$ to the canonical morphism $\bfk_{\Delta\times [0,\infty)}\to \cK_A$, we obtain the identity morphism $\id_{\cK_A}$ of $\cK_A$. 
    Hence (iii) implies (ii). 
\end{proof}
\begin{lemma}\label{lemma:Kernelconst}
    Let $H\in C_{cc}(T^*M)$.
    If $H\equiv c$ on some open neighborhood of $A$ then 
    $\cK_H\ostar\cK_A\simeq T_c\cK_A$. 
\end{lemma}
\begin{proof}
    We may assume $c=0$ since $\cK_H\simeq T_c\cK_{H-c}$. 
    Moreover we may assume $H$ is smooth since the kernel $\cK_H$ for continuous $H$ is defined as a colimit of $\cK_{H'}$ for smooth $H'$ and the colimit commute with $\ostar$. 

    Since $H\sharp(\mathchar`-) \colon C_{cc}^\infty(T^*M\times [0,1],A)\to  C_{cc}^\infty(T^*M\times [0,1],A)$ is an order preserving bijection, 
    the induced functor $\cC_{cc}^\infty(T^*M\times [0,1],A)\to  \cC_{cc}^\infty(T^*M\times [0,1],A)$ is an equivalence. 
    In particular, it is a final functor. 
    By the commutativity of $\colim$ and $\ostar$, we obtain
    \begin{equation}
    \begin{aligned}
        \cK_H\ostar\cK_A
        & \simeq \cK_H \ostar \colim_{H'\in \cC_{cc}^\infty(T^*M\times [0,1],A)} \cK_{H'} \\
        & \simeq \colim_{H'\in \cC_{cc}^\infty(T^*M\times [0,1],A)}\cK_{H\sharp H'} \\
        & \simeq \colim_{H''\in \cC_{cc}^\infty(T^*M\times [0,1],A)}\cK_{H''}
        \simeq \cK_A.
    \end{aligned}
    \end{equation}
\end{proof}

\begin{remark}\label{remark:OSrem}
    The main result of \cite{OnoSugimoto} restricted for cotangent bundles is very similar to \cref{theorem:Hamchar}. 
    Their relative symplectic cohomologies are defined as some colimits of Hamiltonian Floer cohomologies. 
    In other words, the colimit is taken at the Hom-space level. 
    On the other hand, our colimit is taken in the sheaf category $\cT(T^*M\times T^*M)$. 
    Due to the fact that Hom and colim do not commute in general, 
    the direct relationship between our results and theirs is not clear. 
    The author expect that colimit in this paper, being taken in the sheaf category, would be easier to study.

    Since the relation between sheaf kernels and symplectic cohomologies of domains in cotangent bundles are described in \cite{KSZ23},
    the author expects a concrete relation between the kernel $\cK_A$ and relative symplectic cohomology of $A \subset T^*M$. 
\end{remark}

For $\zeta_{\mathrm{Ham}}$-heavy subset $A\subset T^*M$, we define $\zeta_{\mathrm{Ham}, A}\colon C_{cc}(M)\to \bR$ by
\begin{equation}
    \zeta_{\mathrm{Ham}, A}(H) \coloneqq -\lim_k\frac{c(\id_{\cK_A};\cK_A, \cK_{H}^{\ostar -k}( \cK_A))}{k}.
\end{equation}
This is a special case of above construction of $\zeta_{\mathrm{Ham},e}$. 
At the end of this appendix, we introduce the basic properties of $\zeta_{\mathrm{Ham},A}$. 
\begin{proposition}
    For $\zeta_{\mathrm{Ham}}$-heavy subset $A\subset T^*M$, we have the following:
    \begin{enumerate}
        \item $\zeta_{\mathrm{Ham}, A}$ is a partial symplectic quasi-state. 
        \item For every $H\in C_{cc}(T^*M)$, $\zeta_{\mathrm{Ham}, A}(H)\le \zeta_{\mathrm{Ham}}(H)$. 
        \item $A$ is $\zeta_{\mathrm{Ham}, A}$-superheavy. 
    \end{enumerate}
\end{proposition}
\begin{proof}
    (i) and (ii) are already stated for a more general setting. 

    (iii) is a direct consequence of \cref{lemma:Kernelconst}. 
\end{proof}

\printbibliography

@article {GKS,
    AUTHOR = {Guillermou, St{\'e}phane and Kashiwara, Masaki and Schapira,
              Pierre},
     TITLE = {Sheaf quantization of {H}amiltonian isotopies and applications
              to nondisplaceability problems},
   JOURNAL = {Duke Math. J.},
  FJOURNAL = {Duke Mathematical Journal},
    VOLUME = {161},
      YEAR = {2012},
    NUMBER = {2},
     PAGES = {201--245},
      ISSN = {0012-7094},
     CODEN = {DUMJAO},
   MRCLASS = {53D35 (18F30)},
  MRNUMBER = {2876930},
MRREVIEWER = {Corrado Marastoni},
       DOI = {10.1215/00127094-1507367},
       URL = {http://dx.doi.org/10.1215/00127094-1507367},
}

@book {KS90,
    AUTHOR = {Kashiwara, Masaki and Schapira, Pierre},
     TITLE = {Sheaves on manifolds},
    SERIES = {Grundlehren der Mathematischen Wissenschaften},
    VOLUME = {292},
 PUBLISHER = {Springer-Verlag, Berlin},
      YEAR = {1990},
     PAGES = {x+512},
      ISBN = {3-540-51861-4},
   MRCLASS = {58G07 (18F20 32C38 35A27)},
  MRNUMBER = {1299726 (95g:58222)},
}

@preamble{
   "\def\cprime{$'$} "
}

@incollection {GS14,
    AUTHOR = {Guillermou, St{\'e}phane and Schapira, Pierre},
     TITLE = {Microlocal theory of sheaves and {T}amarkin's non
              displaceability theorem},
 BOOKTITLE = {Homological mirror symmetry and tropical geometry},
    SERIES = {Lect. Notes Unione Mat. Ital.},
    VOLUME = {15},
     PAGES = {43--85},
 PUBLISHER = {Springer, Cham},
      YEAR = {2014},
   MRCLASS = {53D37 (35A27 37J05)},
  MRNUMBER = {3330785},
MRREVIEWER = {Hansol Hong},
       DOI = {10.1007/978-3-319-06514-4_3},
       URL = {http://dx.doi.org/10.1007/978-3-319-06514-4_3},
}

@article{KS18persistent,
	title={Persistent homology and microlocal sheaf theory},
	author={Kashiwara, Masaki and Schapira, Pierre},
	journal={Journal of Applied and Computational Topology},
	volume={2},
	number={1-2},
	pages={83--113},
	year={2018},
	publisher={Springer}
}

@InProceedings{Tamarkin,
	author="Tamarkin, Dmitry",
	editor="Hitrik, Michael
	and Tamarkin, Dmitry
	and Tsygan, Boris
	and Zelditch, Steve",
	title="Microlocal Condition for Non-displaceability",
	booktitle="Algebraic and Analytic Microlocal Analysis",
	year="2018",
	publisher="Springer International Publishing",
	address="Cham",
	pages="99--223"
}

@article {Ike19,
    AUTHOR = {Ike, Yuichi},
     TITLE = {Compact exact {L}agrangian intersections in cotangent bundles
              via sheaf quantization},
   JOURNAL = {Publ. Res. Inst. Math. Sci.},
  FJOURNAL = {Publications of the Research Institute for Mathematical
              Sciences},
    VOLUME = {55},
      YEAR = {2019},
    NUMBER = {4},
     PAGES = {737--778},
      ISSN = {0034-5318},
   MRCLASS = {14F05},
  MRNUMBER = {4024997},
       DOI = {10.4171/PRIMS/55-4-3},
       URL = {https://doi.org/10.4171/PRIMS/55-4-3},
}

@misc{Vi19,
  title={Sheaf Quantization of Lagrangians and Floer cohomology}, 
  author={Claude Viterbo},
  year={2019},
  eprint={1901.09440},
  archivePrefix={arXiv},
  primaryClass={math.SG}
}

@misc{Gu12,
      title={Quantization of conic Lagrangian submanifolds of cotangent bundles}, 
      author={Stéphane Guillermou},
      year={2012},
      eprint={1212.5818},
      archivePrefix={arXiv},
      primaryClass={math.SG}
}

@article {Gu23cotangent,
    AUTHOR = {Guillermou, St\'{e}phane},
     TITLE = {Sheaves and symplectic geometry of cotangent bundles},
   JOURNAL = {Ast\'{e}risque},
  FJOURNAL = {Ast\'{e}risque},
    NUMBER = {440},
      YEAR = {2023},
     PAGES = {x+274},
      ISSN = {0303-1179,2492-5926},
      ISBN = {978-2-85629-972-2},
   MRCLASS = {53D12 (18F20 35A27)},
  MRNUMBER = {4612528},
       DOI = {10.24033/ast.1199},
       URL = {https://doi.org/10.24033/ast.1199},
}

@article{AI20,
	title={Persistence-like distance on {T}amarkin's category and symplectic displacement energy},
	author={Asano, Tomohiro and Ike, Yuichi},
	JOURNAL = {J. Symplectic Geom.},
	FJOURNAL = {The Journal of Symplectic Geometry},
	VOLUME = {18},
	YEAR = {2020},
	NUMBER = {3},
	PAGES = {613--649},
}

@article{bokstedt1993homotopy,
  title={Homotopy limits in triangulated categories},
  author={B{\"o}kstedt, Marcel and Neeman, Amnon},
  journal={Compositio Mathematica},
  volume={86},
  number={2},
  pages={209--234},
  year={1993}
}

@misc{guillermou2022gamma,
  title={The singular support of sheaves is $\gamma$-coisotropic},
  author={Guillermou, St{\'e}phane and Viterbo, Claude},
  eprint = {2203.12977},
	eprinttype = {arXiv},
	primaryclass = {math.SG},
  year={2022}
}

@article{FS07,
    AUTHOR = {Frauenfelder, Urs and Schlenk, Felix},
     TITLE = {Hamiltonian dynamics on convex symplectic manifolds},
   JOURNAL = {Israel J. Math.},
  FJOURNAL = {Israel Journal of Mathematics},
    VOLUME = {159},
      YEAR = {2007},
     PAGES = {1--56},
      ISSN = {0021-2172},
   MRCLASS = {53D40 (37J05 37J45)},
  MRNUMBER = {2342472},
MRREVIEWER = {Michael J. Usher},
       DOI = {10.1007/s11856-007-0037-3},
       URL = {https://doi.org/10.1007/s11856-007-0037-3},
}

@article{AI22completeness,
  title={Completeness of derived interleaving distances and sheaf quantization of non-smooth objects}, 
  author={Tomohiro Asano and Yuichi Ike},
  journal={Mathematische Annalen},  
  year={2024}
}

@misc{viterbo2022supports,
      title={On the supports in the Humili\`ere completion and $\gamma$-coisotropic sets}, 
      author={Claude Viterbo},
      year={2022},
      eprint={2204.04133},
      archivePrefix={arXiv},
      primaryClass={math.SG}
}

@article {KO22rigidfiber,
    AUTHOR = {Kawasaki, Morimichi and Orita, Ryuma},
     TITLE = {Rigid fibers of integrable systems on cotangent bundles},
   JOURNAL = {J. Math. Soc. Japan},
  FJOURNAL = {Journal of the Mathematical Society of Japan},
    VOLUME = {74},
      YEAR = {2022},
    NUMBER = {3},
     PAGES = {829--847},
      ISSN = {0025-5645},
   MRCLASS = {57R17 (53D12 53D20 53D35 53D40 58K05 70H06)},
  MRNUMBER = {4484232},
       DOI = {10.2969/jmsj/84278427},
       URL = {https://doi.org/10.2969/jmsj/84278427},
}

@article {MVZ12,
    AUTHOR = {Monzner, Alexandra and Vichery, Nicolas and Zapolsky, Frol},
     TITLE = {Partial quasimorphisms and quasistates on cotangent bundles,
              and symplectic homogenization},
   JOURNAL = {J. Mod. Dyn.},
  FJOURNAL = {Journal of Modern Dynamics},
    VOLUME = {6},
      YEAR = {2012},
    NUMBER = {2},
     PAGES = {205--249},
      ISSN = {1930-5311},
   MRCLASS = {53Dxx},
  MRNUMBER = {2968955},
       DOI = {10.3934/jmd.2012.6.205},
       URL = {https://doi.org/10.3934/jmd.2012.6.205},
}

@book{CB15,
    AUTHOR = {Cushman, Richard H. and Bates, Larry M.},
     TITLE = {Global aspects of classical integrable systems},
   EDITION = {Second},
 PUBLISHER = {Birkh\"{a}user/Springer, Basel},
      YEAR = {2015},
     PAGES = {xx+477},
   MRCLASS = {37-01 (37Jxx 37Kxx 70Exx 70H06)},
  MRNUMBER = {3242761},
       DOI = {10.1007/978-3-0348-0918-4},
       URL = {https://doi.org/10.1007/978-3-0348-0918-4},
}

@misc{guillermou2022viterbos,
      title={Viterbo's spectral bound conjecture for homogeneous spaces}, 
      author={Stéphane Guillermou and Nicolas Vichery},
      year={2022},
      eprint={2203.13700},
      archivePrefix={arXiv},
      primaryClass={math.SG}
}

@misc{viterbo2022inverse,
      title={Inverse reduction inequalities for spectral numbers and applications}, 
      author={Claude Viterbo},
      year={2022},
      eprint={2203.13172},
      archivePrefix={arXiv},
      primaryClass={math.SG}
}

@article{shelukhin2022viterbo,
  title={Viterbo conjecture for Zoll symmetric spaces},
  author={Shelukhin, Egor},
  journal={Inventiones mathematicae},
  volume={230},
  number={1},
  pages={321--373},
  year={2022},
  publisher={Springer}
}

@article{shelukhin2022symplectic,
  title={Symplectic cohomology and a conjecture of {V}iterbo},
  author={Shelukhin, Egor},
  journal={Geometric and Functional Analysis},
  volume={32},
  number={6},
  pages={1514--1543},
  year={2022},
  publisher={Springer}
}

@article {viterbo2022symplectic,
    AUTHOR = {Viterbo, Claude},
     TITLE = {Symplectic homogenization},
   JOURNAL = {J. \'{E}c. polytech. Math.},
  FJOURNAL = {Journal de l'\'{E}cole polytechnique. Math\'{e}matiques},
    VOLUME = {10},
      YEAR = {2023},
     PAGES = {67--140},
      ISSN = {2429-7100,2270-518X},
   MRCLASS = {37J06 (35F20 49J45 49L25 53D35)},
  MRNUMBER = {4536296},
       DOI = {10.5802/jep.214},
       URL = {https://doi.org/10.5802/jep.214},
}

@article{EP09rigid,
  title={Rigid subsets of symplectic manifolds},
  author={Entov, Michael and Polterovich, Leonid},
  journal={Compositio Mathematica},
  volume={145},
  number={3},
  pages={773--826},
  year={2009},
  publisher={London Mathematical Society}
}

@incollection {OnoSugimoto,
    AUTHOR = {Ono, Kaoru and Sugimoto, Yoshihiro},
     TITLE = {Note on (super) heavy subsets in symplectic manifolds},
 BOOKTITLE = {Gromov-{W}itten theory, gauge theory and dualities},
    SERIES = {Proc. Centre Math. Appl. Austral. Nat. Univ.},
    VOLUME = {48},
     PAGES = {174--192},
 PUBLISHER = {Austral. Nat. Univ., Canberra},
      YEAR = {2019},
      ISBN = {978-0-6481056-2-6},
   MRCLASS = {53D05 (53D40)},
  MRNUMBER = {3951405},
MRREVIEWER = {Erkao\ Bao},
}

@article {Lanzat13,
    AUTHOR = {Lanzat, Sergei},
     TITLE = {Quasi-morphisms and symplectic quasi-states for convex
              symplectic manifolds},
   JOURNAL = {Int. Math. Res. Not. IMRN},
  FJOURNAL = {International Mathematics Research Notices. IMRN},
      YEAR = {2013},
    NUMBER = {23},
     PAGES = {5321--5365},
      ISSN = {1073-7928,1687-0247},
   MRCLASS = {53D40},
  MRNUMBER = {3142258},
MRREVIEWER = {Michael\ J.\ Usher},
       DOI = {10.1093/imrn/rns205},
       URL = {https://doi.org/10.1093/imrn/rns205},
}

@article{AGHIV,
	author = {Tomohiro Asano and St\'ephane Guillermou and Vincent Humili\`ere and Yuichi Ike and Claude Viterbo},
     title = {The $\gamma $-support as a micro-support},
     journal = {Comptes Rendus. Math\'ematique},
     pages = {1333--1340},
     publisher = {Acad\'emie des sciences, Paris},
     volume = {361},
     year = {2023},
     doi = {10.5802/crmath.499}
}

@article {Oh97,
    AUTHOR = {Oh, Yong-Geun},
     TITLE = {Symplectic topology as the geometry of action functional. {I}.
              {R}elative {F}loer theory on the cotangent bundle},
   JOURNAL = {J. Differential Geom.},
  FJOURNAL = {Journal of Differential Geometry},
    VOLUME = {46},
      YEAR = {1997},
    NUMBER = {3},
     PAGES = {499--577},
      ISSN = {0022-040X,1945-743X},
   MRCLASS = {58E05 (57R70 58D15 58F05)},
  MRNUMBER = {1484890},
MRREVIEWER = {Joa\ Weber},
       URL = {http://projecteuclid.org/euclid.jdg/1214459976},
}

@article {Oh99,
    AUTHOR = {Oh, Yong-Geun},
     TITLE = {Symplectic topology as the geometry of action functional.
              {II}. {P}ants product and cohomological invariants},
   JOURNAL = {Comm. Anal. Geom.},
  FJOURNAL = {Communications in Analysis and Geometry},
    VOLUME = {7},
      YEAR = {1999},
    NUMBER = {1},
     PAGES = {1--54},
      ISSN = {1019-8385,1944-9992},
   MRCLASS = {53D40 (57R58)},
  MRNUMBER = {1674121},
MRREVIEWER = {Darko\ Milinkovi\'{c}},
       DOI = {10.4310/CAG.1999.v7.n1.a1},
       URL = {https://doi.org/10.4310/CAG.1999.v7.n1.a1},
}

@article {Kuowrap,
    AUTHOR = {Kuo, Christopher},
     TITLE = {Wrapped sheaves},
   JOURNAL = {Adv. Math.},
  FJOURNAL = {Advances in Mathematics},
    VOLUME = {415},
      YEAR = {2023},
     PAGES = {Paper No. 108882, 71},
      ISSN = {0001-8708,1090-2082},
   MRCLASS = {18F20 (32C38 53D37 53D40 55N30)},
  MRNUMBER = {4543073},
MRREVIEWER = {Mee\ Seong\ Im},
       DOI = {10.1016/j.aim.2023.108882},
       URL = {https://doi.org/10.1016/j.aim.2023.108882},
}

@article{MSV23,
author = {Mak, Cheuk Yu and Sun, Yuhan and Varolgunes, Umut},
title = {A characterization of heaviness in terms of relative symplectic cohomology},
journal = {Journal of Topology},
volume = {17},
number = {1},
pages = {e12327},
doi = {https://doi.org/10.1112/topo.12327},
url = {https://londmathsoc.onlinelibrary.wiley.com/doi/abs/10.1112/topo.12327},
Zxeprint = {https://londmathsoc.onlinelibrary.wiley.com/doi/pdf/10.1112/topo.12327},
year = {2024}
}

@misc{VicheryHDC,
      title={Homological differential calculus}, 
      author={Nicolas Vichery},
      year={2013},
      eprint={1310.4845},
      archivePrefix={arXiv},
      primaryClass={math.AT}
}

@book {HZ11,
    AUTHOR = {Hofer, Helmut and Zehnder, Eduard},
     TITLE = {Symplectic invariants and {H}amiltonian dynamics},
    SERIES = {Modern Birkh\"{a}user Classics},
      NOTE = {Reprint of the 1994 edition},
 PUBLISHER = {Birkh\"{a}user Verlag, Basel},
      YEAR = {2011},
     PAGES = {xiv+341},
      ISBN = {978-3-0348-0103-4},
   MRCLASS = {53D35 (37J05 53D40 70G45 70H05)},
  MRNUMBER = {2797558},
       DOI = {10.1007/978-3-0348-0104-1},
       URL = {https://doi.org/10.1007/978-3-0348-0104-1},
}

@article {BGT13,
    AUTHOR = {Blumberg, Andrew J. and Gepner, David and Tabuada,
              Gon\c{c}alo},
     TITLE = {A universal characterization of higher algebraic {$K$}-theory},
   JOURNAL = {Geom. Topol.},
  FJOURNAL = {Geometry \& Topology},
    VOLUME = {17},
      YEAR = {2013},
    NUMBER = {2},
     PAGES = {733--838},
      ISSN = {1465-3060,1364-0380},
   MRCLASS = {19D10 (18D20 19D25 19D55 55N15 55U40)},
  MRNUMBER = {3070515},
MRREVIEWER = {Ross\ Staffeldt},
       DOI = {10.2140/gt.2013.17.733},
       URL = {https://doi.org/10.2140/gt.2013.17.733},
}

@article {EP06quasi-state,
    AUTHOR = {Entov, Michael and Polterovich, Leonid},
     TITLE = {Quasi-states and symplectic intersections},
   JOURNAL = {Comment. Math. Helv.},
  FJOURNAL = {Commentarii Mathematici Helvetici. A Journal of the Swiss
              Mathematical Society},
    VOLUME = {81},
      YEAR = {2006},
    NUMBER = {1},
     PAGES = {75--99},
      ISSN = {0010-2571,1420-8946},
   MRCLASS = {53D35 (46L30 53D40)},
  MRNUMBER = {2208798},
MRREVIEWER = {Martin\ Pinsonnault},
       DOI = {10.4171/CMH/43},
       URL = {https://doi.org/10.4171/CMH/43},
}

@phdthesis{VicheryPhD,
	author = {Vichery, Nicolas},
	school = {{\'E}cole polytechnique},
	title = {Homog{\'e}n{\'e}isation symplectique et Applications de la th{\'e}orie des faisceaux {\`a} la topologie symplectique},
	url = {https://tel.archives-ouvertes.fr/pastel-00780016/},
	year = {2012}}

@misc{KSZ23,
      title={On the Hochschild cohomology of Tamarkin categories}, 
      author={Christopher Kuo and Vivek Shende and Bingyu Zhang},
      year={2023},
      eprint={2312.11447},
      archivePrefix={arXiv},
      primaryClass={math.SG}
}

@book{LurieHTT,
    AUTHOR = {Lurie, Jacob},
     TITLE = {Higher topos theory},
    SERIES = {Annals of Mathematics Studies},
    VOLUME = {170},
 PUBLISHER = {Princeton University Press, Princeton, NJ},
      YEAR = {2009},
     PAGES = {xviii+925},
   MRCLASS = {18-02 (18B25 18E35 18G30 18G55 55U40)},
  MRNUMBER = {2522659},
MRREVIEWER = {Mark\ Hovey},
       DOI = {10.1515/9781400830558},
       URL = {https://doi.org/10.1515/9781400830558},
}

@misc{LurieHA,
    title={Higher Algebra},
    author={Lurie, Jacob},
    URL = {https://www.math.ias.edu/~lurie/papers/HA.pdf}
}

@misc{AHV24,
      title={Higher Dimensional Birkhoff attractors}, 
      author={Marie-Claude Arnaud and Vincent Humilière and Claude Viterbo},
      year={2024},
      eprint={2404.00804},
      archivePrefix={arXiv},
      primaryClass={math.SG}
}

@article {NS18,
    AUTHOR = {Nikolaus, Thomas and Scholze, Peter},
     TITLE = {On topological cyclic homology},
   JOURNAL = {Acta Math.},
  FJOURNAL = {Acta Mathematica},
    VOLUME = {221},
      YEAR = {2018},
    NUMBER = {2},
     PAGES = {203--409},
      ISSN = {0001-5962,1871-2509},
   MRCLASS = {55U35 (16E40 18E30 19D99)},
  MRNUMBER = {3904731},
MRREVIEWER = {Geoffrey\ M. L. Powell},
       DOI = {10.4310/ACTA.2018.v221.n2.a1},
       URL = {https://doi.org/10.4310/ACTA.2018.v221.n2.a1},
}

@inproceedings {Entov14,
    AUTHOR = {Entov, Michael},
     TITLE = {Quasi-morphisms and quasi-states in symplectic topology},
 BOOKTITLE = {Proceedings of the {I}nternational {C}ongress of
              {M}athematicians---{S}eoul 2014. {V}ol. {II}},
     PAGES = {1147--1171},
 PUBLISHER = {Kyung Moon Sa, Seoul},
      YEAR = {2014},
   MRCLASS = {53D35 (17B99 20F69 46L30 53D40 53D45)},
  MRNUMBER = {3728656},
MRREVIEWER = {Karl\ Friedrich\ Siburg},
}

@book {PR14,
    AUTHOR = {Polterovich, Leonid and Rosen, Daniel},
     TITLE = {Function theory on symplectic manifolds},
    SERIES = {CRM Monograph Series},
    VOLUME = {34},
 PUBLISHER = {American Mathematical Society, Providence, RI},
      YEAR = {2014},
     PAGES = {xii+203},
      ISBN = {978-1-4704-1693-5},
   MRCLASS = {53Dxx (22E65 57R17 57R58 58E05 81P45)},
  MRNUMBER = {3241729},
MRREVIEWER = {Charles-Michel\ Marle},
       DOI = {10.1090/crmm/034},
       URL = {https://doi.org/10.1090/crmm/034},
}

@misc{viterbo2018functors,
      title={Functors and Computations in Floer homology with Applications Part II}, 
      author={Viterbo, Claude},
      year={2018},
      eprint={1805.01316},
      archivePrefix={arXiv},
      primaryClass={math.SG}
}

\noindent Tomohiro Asano: 
Research Institute for Mathematical Sciences, Kyoto University, \linebreak Kitashirakawa-Oiwake-Cho, Sakyo-ku, 606-8502, Kyoto, Japan.

\noindent \textit{E-mail address}: \texttt{tasano[at]kurims.kyoto-u.ac.jp}, \texttt{tomoh.asano[at]gmail.com}

\end{document}